\documentclass[%
 reprint,
 amsmath,amssymb,
aps,
prd,
]{revtex4-1}

\usepackage{algorithm,algpseudocode}
\usepackage{paralist}
\usepackage{latexsym,graphicx,epsfig,amsmath,amsfonts,amssymb,bm}
\usepackage{relsize}
\usepackage{exscale}
\usepackage{epstopdf}
\usepackage{subfigure}
\usepackage{psfrag}
\usepackage{mathrsfs}
\usepackage{amsthm}
\usepackage{booktabs}
\usepackage{array}
\usepackage{graphicx}% Include figure files
\usepackage{dcolumn}% Align table columns on decimal point
\usepackage{bm}% bold math
\allowdisplaybreaks
\newtheorem{definition}{Definition}
\newtheorem{lemma}{Lemma}

\newtheorem{theorem}{Theorem}
\newtheorem{corollary}{Corollary}

\def\dfr#1#2{\displaystyle{\frac{#1}{#2}}}
\def\pt{\partial}
\def\al{\alpha}

\begin{document}

\preprint{APS/123-QED}

\title{Design of Provably Physical-Constraint-Preserving Methods \\
for General Relativistic Hydrodynamics}

%\affiliation{%

\author{Kailiang Wu}
\email{wu.3423@osu.edu}
\affiliation{
Department of Mathematics, The Ohio State University, Columbus, OH 43210, USA
}%
%\affiliation{
%\affiliation{%
\date{\today}% It is always \today, today,

\begin{abstract}
The paper develops high-order physical-constraint-preserving (PCP) methods for general relativistic hydrodynamic (GRHD) equations, equipped with a general equation of state. Here the physical constraints, describing the admissible states of GRHD, are referred to
the subluminal constraint on the fluid velocity and the positivity of the density, pressure and specific internal energy.
Preserving these constraints is very important for robust computations, otherwise violating one of them
will lead to the ill-posed problem and numerical instability.
To overcome the difficulties arising from the inherent strong nonlinearity contained in the constraints, we derive an equivalent definition of the admissible states. Using this definition, we prove the convexity, scaling invariance and Lax-Friedrichs (LxF) splitting property of the admissible state set $\mathcal G$, and discover the dependence of $\mathcal G$ on the spacetime metric. Unfortunately, such dependence yields the non-equivalence of $\mathcal G$ at different points in curved spacetime, and invalidates the convexity of $\mathcal G$ in analyzing PCP schemes.
This obstacle is effectively overcame
by introducing a new formulation of the GRHD equations. Based on this formulation and the above theories,
a first-order LxF scheme is designed on general unstructured mesh and rigorously proved to be PCP under a CFL condition.
With two types of PCP limiting procedures, we design high-order, {\em provably} (not probably) PCP methods under discretization on
the proposed new formulation. These high-order methods include the PCP finite difference, finite volume and discontinuous Galerkin methods.
%\begin{description}
%\end{description}
\end{abstract}

\maketitle

%\tableofcontents

\section{Introduction}

In many cases, high energy physics and astrophysics may involve hydrodynamical problems with special or general relativistic effect, corresponding to that the fluid flow is at nearly speed of light, or the influence of strong gravitational field on the hydrodynamics cannot be neglected.
Relativistic hydrodynamics (RHD) is very important in investigating a number of astrophysical scenarios from stellar to galactic scales, e.g. astrophysical jets, gamma-ray bursts, core collapse super-novae, formation of black holes, merging of compact binaries, etc.

The governing equations of RHDs are highly nonlinear, making their analytical
treatment extremely difficult. Numerical simulation has become a primary and powerful approach to understand the physical mechanisms in the RHDs.
%For example, numerical relativity recently played an important role in the discovery of gravitational waves [xxx] and helped
% confirm that the signal extracted from the LIGO data indeed originated from binary black hole mergers.
The pioneering numerical work on the RHD equations may date back to the Lagrangian finite difference code via artificial viscosity
for the spherically symmetric GRHD equations \cite{May1966,May1967}.
Wilson \cite{Wilson:1972} first attempted to solve multiple-dimensional RHD equations by using the Eulerian finite difference method with the artificial viscosity technique.
Since the 1990s, the numerical study of RHD has attracted considerable attention, and various modern shock-capturing methods based on Riemann solvers have been developed for the RHD equations.
The readers are referred to the early review articles \cite{MME:2003,Font2003,Font2008,Marti2015}
and some more recent works e.g. \cite{WuTang2014,WuTang2016GRP,Bugner2016} as well as references therein.%,ZhaoTang2016

Most existing methods do not preserve the positivity of the density, pressure and the specific internal
energy as well as the bound of the fluid velocity, %at the same time,
although they have been used to solve some RHD problems successfully.
There exists a big risk of failure when a numerical scheme
is applied to the RHD problems involving large Lorentz
factor, low density or pressure, or strong discontinuity. This is
because once the negative
density/pressure or the superluminal
fluid velocity is obtained during numerical simulations, the eigenvalues
of the Jacobian matrix become imaginary so that the discrete problem becomes ill-posed.
Moreover, the superluminal
fluid velocity also yields imaginary Lorentz factor and leads to the violation of the relativistic causality.
It is therefore significative to design high-order numerical schemes, whose solutions
 satisfy the intrinsic physical constraints.

Recent years have witnessed some advances in
developing high-order bound-preserving type schemes for hyperbolic conservation laws.
Those schemes are mainly built on two types of limiting procedures.
One is the simple scaling limiting procedure
for the reconstructed or evolved solution polynomials in
a finite volume or discontinuous Galerkin (DG) method, see e.g. \cite{zhang2010,zhang2010b,zhang2011,Xing2010,zhang2012a,Endeve2015,XXzhang2016}.
Another is the flux-corrected limiting procedure,
which can be used on high-order finite difference, finite volume
and DG methods, see e.g. \cite{Xu_MC2013,Hu2013,Liang2014,JiangXu2013,XiongQiuXu2014,Christlieb,Christlieb2016}.
A survey of the maximum-principle-satisfying or positivity-preserving
 high-order schemes based on the first type limiter was presented in \cite{zhang2011b}.
 The readers are also referred to \cite{xuzhang2016} for a review of
 these two approaches.
Recently, by extending the above bound-preserving techniques, two types of physical-constraint-preserving (PCP) schemes
were developed for the special RHD equations with an ideal equation of state (EOS), i.e., the high-order PCP finite
difference WENO (weighted essentially non-oscillatory) schemes  \cite{WuTang2015} and the bound-preserving DG methods \cite{QinShu2016}.
More recently, the high-order PCP central DG methods were proposed in \cite{WuTang2016gEOS} for special RHD with a general EOS. The extension of PCP schemes to the ideal relativistic magnetohydrodynamics was studied in \cite{WuTang2016},
where the importance of divergence-free magnetic fields in achieving PCP methods was revealed in theory  for the first time.

The aim of this paper is to design high-order, {\em provably} PCP methods for the GRHD equations
with a general EOS, including PCP finite difference, finite volume and DG methods.
Developing {\em provably} PCP methods for GRHD with a general EOS is very nontrivial and still untouched in literature.
%There are mainly three difficulties: (1).
The technical challenges mainly come from three aspects:
(1). The inherent nonlinear coupling between the GRHD equations due to the Lorentz factor,
curved spacetime and general EOS, e.g., the dearth of explicit expression of the primitive variables and
flux vectors with respect to the conservative/state vector. (2). One more physical constraint for the fluid velocity in addition to the positivity of density, pressure and specific internal energy. (3). The non-equivalence of the admissible state sets defined at different points in curved spacetime.
It is noticed in \cite{Radice2014} that Redice, Rezzolla and Galeazzi once attempted to extend the flux-corrected limiter in non-relativistic case \cite{Hu2013}
to the GRHD equations, but only achieved enforcing the
positivity of density. The importance as well as the difficulty of designing completely PCP schemes were also mentioned in \cite{Rezzolla2013,Radice2014}.
The work in this paper overcomes the above difficulties, via a new formulation of the GRHD equations and rigourously theoretical
analysis on the admissible states of GRHD.

The paper is organized as follows. Sec. \ref{sec:governingequ} introduces
 the governing equations of GRHD and the EOS.
Sec. \ref{sec:admis-state} derives
several properties of the admissible state set and
proposes a new formulation of the GRHD equations, which play pivotal roles in designing provably PCP methods.
Sec. \ref{sec:FRscheme} proves the PCP property of the first-order LxF scheme on general unstructured mesh.
High-order, provably PCP methods are presented in Sec. \ref{sec:HIscheme} with detailed implementation procedures, including
 PCP finite volume and DG methods in Sec. \ref{sec:FVDG} and
 PCP finite difference methods in Sec. \ref{sec:FD}.
 Concluding remarks are presented in Sec. \ref{sec:conclude}.
%concludes the paper with several remarks.
For better legibility, %we summarize the widely-used nations in Appendix \ref{app:notations},
%and
%we put
all the proofs of the lemmas and theorems are put in Appendix \ref{app:proofs}.

Throughout the paper, we use a spacetime signature $(-,+,+,+)$ with Greek indices running from 0 to 3 and Latin indices from 1 to 3.
We also employ the Einstein summation convention over repeated indices, and the geometrized unit system
so that the speed of light in vacuum %$c$
and the gravitational constant are equal to one.

\section{Governing equations} \label{sec:governingequ}

%\subsection{Covariant form}

The general relativistic hydrodynamic (GRHD) equations \cite{Font2008} consist of the local conservation laws of the baryon number density
and the stress-energy tensor $T^{\mu\nu}$,
\begin{align}\label{eq:J}
&\nabla_{\mu} \big(\rho u^\mu\big) = 0, \\
&\nabla_{\mu} T^{\mu\nu} = 0, \label{eq:T}
\end{align}
where $\rho$ denotes the rest-mass density, $u^{\mu}$ represents the fluid four-velocity,
%the indexes $\mu$ and $\nu$ run from 0 to 3,
and $\nabla_{\mu}$ stands for the covariant derivative associated with the four-dimensional spacetime
metric $g_{\mu\nu}$, i.e., the line element in four-dimensional spacetime is $ds^2=g_{\mu\nu}d x^{\mu} d x^{\nu}$.
The stress-energy tensor for an ideal fluid is defined by
$$
T^{\mu\nu} = \rho h  u^\mu u^\nu + p g^{\mu \nu},
$$
where $p$ denotes the pressure, and $h$ represents the specific enthalpy defined by
\begin{equation*}%\label{eq:Defh}
h = 1+e+{p}/{\rho},
\end{equation*}
with $e$ denoting the specific internal energy.
%, and $g^{\mu \lambda} g_{\lambda \nu} = \delta_{\nu}^\mu$ with $\delta_{\nu}^\mu$ denoting the Kronecker symbol.

%\subsection{Equation of state}

An additional equation for the thermodynamical variables, i.e. the so-called equation of
state (EOS), is needed to close the system \eqref{eq:J}--\eqref{eq:T}. %for given spacetime.
In general, the EOS can be expressed as
\begin{equation}
\label{eq:EOS:e}
e = e(p,\rho),
\end{equation}
or
\begin{equation}
\label{eq:EOS:h}
h =h(p,\rho)= 1+e(p,\rho) + p/\rho.
\end{equation}
The relativistic kinetic theory reveals \cite{WuTang2016gEOS} that a general EOS \eqref{eq:EOS:h} should satisfy
\begin{equation}\label{eq:hcondition1}
h(p,\rho) \ge \sqrt{1+p^2/\rho^2}+p/\rho,
\end{equation}
which is weaker than the condition proposed in \cite{Taub}.
This paper focuses on the causal EOS. We also assume that the fluid's coefficient of thermal expansion is positive, which is valid
for  most of compressible fluids, e.g. the gases. Then the following inequality holds \cite{WuTang2016gEOS}
\begin{equation}\label{eq:gEOSC}
h\left(\frac1{\rho} - \frac{\pt h(p,\rho)}{\pt p} \right) < \frac{\pt h(p,\rho)}{\pt \rho} < 0.
\end{equation}
The most commonly used EOS, called the ideal EOS, is
\begin{equation}\label{eq:iEOS}
h = 1 + \frac{\Gamma p }{(\Gamma -1)\rho},
\end{equation}
with  $\Gamma \in (1,2]$ denoting the adiabatic index. The ideal EOS \eqref{eq:iEOS} and most of the other EOS reported
in numerical RHDs, see e.g. \cite{Mathews,Mignoneetal:2005,Ryu,WuTang2016gEOS}, usually satisfy the conditions \eqref{eq:hcondition1}--\eqref{eq:gEOSC}, and that the function $e(p,\rho)$ is continuously differentiable in ${\mathbb R}^+\times {\mathbb {R}}^+$ with
\begin{equation}\label{eq:epto0}
\mathop{\lim }\limits_{p \to 0^+ } e(p,\rho)=0,\quad \mathop {\lim }\limits_{p \to +\infty } e(p,\rho) =  + \infty,
\end{equation}
for any fixed positive $\rho$.

%\subsection{Spacetime}

In the ``test-fluid'' approximation, where the fluid self-gravity is neglected in comparison to the background gravitational field, the dynamics of
the system is completely governed by Eqs. \eqref{eq:J} and \eqref{eq:T}, together with the EOS \eqref{eq:EOS:h}.
%In those situations where
When such an approximation does not hold, the GRHD equations must be solved
in conjunction with the Einstein gravitational field equations,
%\begin{equation}\label{eq:Einstein}
%R^{\mu\nu} - \frac12 g^{\mu\nu} R = 8 \pi T^{\mu \nu},
%\end{equation}
which %describe the evolution of the geometry in a dynamic spacetime.
 relate the curvature
of spacetime to the distribution of mass-energy. %Here $R^{\mu\nu}$ and $R$ denote the Ricci tensor and the scalar curvature, respectively.

In this paper, we only focus on the numerical methods for the GRHD equations \eqref{eq:J}, \eqref{eq:T} and \eqref{eq:EOS:h},
assuming that the spacetime metric $g_{\mu\nu}$ and its derivatives  $\frac{\partial g_{\mu\nu}}{\partial x^{\delta}}$ are given
or can be numerically computed by a given solver for the Einstein equations %\eqref{eq:Einstein}
in each numerical time-step.
All the following discussions only require that the metric tensor $g_{\mu\nu}$ is real symmetric with signature $(-,+,+,+)$.

%\subsection{3+1 Eulerian formulation}

In order to solve the GRHD equations by using
modern shock-capturing methods, it is more suitable to reformulate
 the covariant form \eqref{eq:J}--\eqref{eq:T} into conservative Eulerian formulation, see e.g. \cite{Marti1991,Banyuls1997}.
For this purpose, we adopt the $3+1$ (ADM) formulation \cite{Arnowitt} to decompose spacetime into a set of non-intersecting spacelike
hypersurfaces with normal $(1/\alpha,-\beta^i/\alpha)$, where $\alpha>0$ is the lapse function and $\beta^i$ is
the shift vector. Within this formalism the spacetime metric $g_{\mu\nu}$ is split as
\begin{equation*}%\label{eq:3add1}
%\begin{split}
ds^2 %= g_{\mu\nu} dx^\mu dx^\nu
 = -(\alpha^2-\beta_i\beta^i) dt^2 + 2 \beta_i dx^{i} dt + \gamma_{ij} d x^i d x^j,
%\end{split}
\end{equation*}
where $\gamma_{ij}$ denotes the 3-metric induced on each spacelike slice and is symmetric positive definite.

Let $g=\det(g_{\mu\nu})$, $\gamma=\det(\gamma_{ij})$
with $\sqrt{-g}=\alpha \sqrt{\gamma}$, and $\Gamma^{\lambda}_{\mu \nu}$ be the Christoffel symbols.
Then, the GRHD equations \eqref{eq:J}--\eqref{eq:T} can be
rewritten as a first-order hyperbolic system \cite{Banyuls1997}
\begin{equation}\label{eq:GRHDcon}
\frac{1}{{\sqrt { - g} }}\left( {\frac{{\partial \sqrt{\gamma}  {\bf U}}}{{\partial {t}}}
+  {\frac{{\partial \sqrt { - g} {{\bf F}^i}({\bf U})}}{{\partial {x^i}}}} } \right) =  {\bf Q} ({\bf U}),
\end{equation}
where
{\small
\begin{equation*}
\begin{aligned}
&{\bf U} =(D,{\bf m},E)^\top, \\
&{\bf F}^i  = \big(D \tilde v^i,  \tilde v^i {\bf m}  +p {\bf e}_i, E \tilde v^i +pv^i \big)^\top,\\
&{\bf Q} = \bigg( 0, T^{\mu\nu} \Big( \frac{\pt g_{\nu j}}{\pt x^{\mu}} - \Gamma^{\delta}_{\nu\mu}g_{\delta j} \Big) ,
\alpha\Big( T^{\mu 0} \frac{\pt \ln \alpha}{\pt x^{\mu}} - T^{\mu \nu} \Gamma^0_{\nu \mu} \Big) \bigg)^\top,\\
%&K_j =T^{\mu\nu} \left( \frac{\pt g_{\nu j}}{\pt x^{\mu}} - \Gamma^{\delta}_{\nu\mu}g_{\delta j} \right),~ j \le 3,
%\\
%&K_4 =\alpha\left( T^{\mu 0} \frac{\pt \ln \alpha}{\pt x^{\mu}} - T^{\mu \nu} \Gamma^0_{\nu \mu} \right),
\end{aligned}
\end{equation*}
}with $\tilde v^i = v^i-{\beta^i}/{\alpha}$,
the mass density $D=\rho W$, the momentum density (row) vector ${\bf m} = \rho h W^2 {\bf v}$, the energy density $E=\rho h W^2 -p$,
and the row vector ${\bf  e}_i$ denoting the $i$-th row of the unit matrix of size 3.
Additionally, row vector ${\bf v}=(v_1,v_2,v_3)$ denotes the 3-velocity of the fluid with the contravariant components defined by
%$v^i=\gamma^{ij}v_j$ is the contravariant 3-velocity tensor defined by
$$v^i = \frac{u^i}{\al u^0}+\frac{\beta^i}{\alpha},$$
and the Lorentz factor $W=\alpha u^0=(1-v^2)^{-\frac12}$ with $v=\sqrt{\gamma_{ij}v^iv^j}=\sqrt{v_jv^j}$.

%${\bf w}=(\rho,{\bf v},p)^\top$ denotes the primitive variable vector,

The physical significance of the solution of \eqref{eq:GRHDcon} and the hyperbolicity of \eqref{eq:GRHDcon} require that the constraints
\begin{equation}\label{eq:constraints:pri}
\rho>0,~ p>0,~ e>0,~ v<1,
\end{equation}
always hold, which can ensure the Jacobian matrix $\pt ( \xi_j {{{\bf F}^j}} )/\pt {\bf U}$ with any $(\xi_1,\xi_2,\xi_3)\neq \bm 0$ has five real eigenvalues,
and five independent real eigenvectors, see e.g. \cite{Font2008,Rezzolla2013}. Specifically, these eigenvalues are
\begin{align*}
& \lambda^{(2)} = \lambda^{(3)} = \lambda^{(4)}
= \xi_j v^j  - \frac{  \xi_j \beta^j }{\alpha},\\
& \lambda^{(3\pm2)}
= \frac{1}{{1 - v^2 c_s^2 }} \bigg\{ \xi_j v^j  (1 - c_s^{ 2} ) \pm c_s W^{-1}
\\
&\quad  \times \sqrt { (1-v^2c_s^2)(\xi_j\xi^j)-(1-c_s^2) (\xi_j v^j)^2  } \bigg\}   - \frac{  \xi_j \beta^j }{\alpha},
\end{align*}
where $c_s$ is the local sound speed defined by \cite{Ryu,ChoiWiita2010}
$$
c_s^2 = h^{-1}  \frac{\pt h(p,\rho)}{\pt \rho}  / \left(\frac1{\rho} - \frac{\pt h(p,\rho)}{\pt p} \right).
$$
It implies $0<c_s<1$ from the condition \eqref{eq:gEOSC} and $c_s=\sqrt{\frac{\Gamma p}{\rho h}}$ for ideal EOS \eqref{eq:iEOS}.

\section{admissible state set} \label{sec:admis-state}

\subsection{Definition and equivalent definition}

For the GRHD equations \eqref{eq:GRHDcon}, it is very natural and intuitive
to define  the (physically) {\emph{admissible state set}} of $\bf  U$ as follows.
\begin{definition}
 The set of admissible states of the GRHD equations
\eqref{eq:GRHDcon} is defined by
\begin{equation}\label{EQ-adm-set01}
\begin{split}
{\mathcal G} &= \Big\{   {\bf U} = (D,{\bf m},E)^\top \big| \rho (  {\bf U}  ) > 0,
\\
&\qquad \qquad p(  {\bf U}  ) > 0, e(  {\bf U}  )>0,~v( {\bf U}  )<1 \Big\}.
\end{split}
\end{equation}
\end{definition}
Unfortunately, it is difficult to verify the four conditions in  \eqref{EQ-adm-set01}
for the given value of $ {\bf U} $,
because there is no explicit expression for the transformation $ {\bf U}  \mapsto  (\rho,p,e,  {\bf v} )^{\rm}$.
This also indicates the difficulty in studying the properties
of  ${\mathcal G}$ and developing the PCP schemes for
\eqref{eq:GRHDcon} with the numerical solution in $\mathcal G$, especially for a general EOS \eqref{eq:EOS:h}.
In practice,  if giving the value of $ {\bf U} $,
then one has to  iteratively solve a nonlinear algebraic equation,
e.g. an equation for  the unknown pressure $p \in \mathbb{R}^+$:
 \begin{equation}\label{eq:solvePgEOS}
  E + p = D h\Big(p, \rho^{[ {\bf U} ]}(p)  \Big) \bigg(1- \frac{m_j m^j}{(E+p)^2}\bigg)^{-\frac12},
 \end{equation}
 where $\rho^{[ {\bf U} ]}(p) = D\sqrt {1 - m_j m^j /{(E + p)^2 }}$. Once the positive solution of the above equation is obtained, denoted by   $p( {\bf U} )$,
   other variables are sequentially calculated by
\begin{equation}\label{eq:solveVRHOgEOS}
\begin{split}
&{ v_j(  {\bf U}  )}  = \frac{{m_j}}{{E + p(  {\bf U}  )}},
\\
&
\rho (  {\bf U}  ) = D\sqrt {1 -  { v_j(  {\bf U}  ) v^j(  {\bf U}  )}   }  ,
\\
& e(  {\bf U}  ) = e( p(  {\bf U}  ),\rho(  {\bf U}  ) ).
\end{split}
\end{equation}

A equivalent simple definition of ${\mathcal G}$ is given as follows, with the proof presented in Appendix \ref{app:1}.

\begin{lemma} \label{lem:equDef}
The admissible state set ${\mathcal G}$ in \eqref{EQ-adm-set01} is equivalent to the following set
\begin{equation}\label{EQ-adm-set02}
\begin{split}
&
{\mathcal G}_\gamma = \Big\{  \left.  {\bf U}  = (D,{\bf m},E)^\top \right| D>0,~
q_\gamma(  {\bf U}  ) >0  \Big\},
\end{split}
\end{equation}
where
$$
q_\gamma(  {\bf U}  )  := E - \sqrt{D^2 + {\bf m}  \bm \Upsilon {\bf m}^\top },
$$
and the matrix $\bm \Upsilon= (\gamma^{ij})_{1\le i,j \le3}$ is positive definite and usually depends on $(t,x^i)$.
\end{lemma}

Based on Lemma \ref{lem:equDef}, the admissible state sets ${\mathcal G}$ and ${\mathcal G}_\gamma$
 will not be deliberately distinguished henceforth.
However, in comparison with  ${\mathcal G}$,
the constraints in the set ${\mathcal G}_\gamma$
are explicit and directly imposed on the conservative variables, so that they can be very easily verified for given value of $\bf U$.

\subsection{Mathematical properties}

With the help of the equivalence between ${\mathcal G} $ and ${\mathcal G}_\gamma $,
 the convexity of admissible state set can then be proved, see Lemma \ref{lam:convexGgamma} with
 proof displayed in Appendix \ref{app:2}.

\begin{lemma}
\label{lam:convexGgamma}
The admissible state set ${\mathcal G}_\gamma$ is an open convex set. Moreover, $\lambda {\bf U}' + (1-\lambda) {\bf U}'' \in {\mathcal G}_\gamma $ for any
${\bf U}' \in {\mathcal G}_\gamma$, ${\bf U}'' \in \overline {\mathcal G}_\gamma$, and  $\lambda \in (0,1]$,
where
$\overline {\mathcal G}_\gamma$ is the closure of ${\mathcal G}_\gamma $.
%with= {\mathcal G}_\gamma \cup \partial {\mathcal G}_\gamma
%$\partial {\mathcal G}_\gamma $ denoting the boundary of ${\mathcal G}_\gamma $.
\end{lemma}

The scaling invariance and Lax-Friedrichs (LxF) splitting properties of ${\mathcal G}_\gamma$ can be further obtained.

\begin{lemma}
\label{lam:propertyG}
Assume that ${\bf U} \in {\mathcal G}_\gamma = {\mathcal G}$, then
\begin{itemize}[\hspace{0em}$\bullet$]
  \item[{\rm(i). (Scaling invariance)}] $\lambda {\bf U} \in {\mathcal G}_\gamma$, for any positive $\lambda$.
  \item[{\rm(ii). (LxF splitting)}]  for any vector $\bm  \xi =(\xi_1,\xi_2,\xi_3)\neq \bm 0$,
  $${\bf U} \pm \varrho_\xi^{-1}{{ \xi_j {\bf F}^j ( {\bf U} )}} \in \overline {\mathcal G}_\gamma,$$
  and
   $${\bf U} \pm {\eta }^{-1}{{ \xi_j {\bf F}^j ( {\bf U} )}} \in {\mathcal G}_\gamma,\quad \mbox{for any}~\eta > \varrho_\xi,$$
   where $\varrho_\xi$ is an appropriate upper bound
  of the spectral radius of the Jacobian matrix $\pt ( \xi_j {{ {\bf  F}^j ( {\bf  U} )}} )/\pt {\bf  U}$. For general EOS, it can be
\begin{equation}\label{eq:WKLvarrhoGEOS}
\varrho_\xi   = \sqrt{\xi_j\xi^j}  + \frac{ | \xi_j \beta^j | }{\alpha}.
\end{equation}
  A smaller/sharper satisfied bound for ideal EOS is
\begin{align*}
& \varrho_\xi   = \frac{1}{ {{1 - v^2 c_s^2 }}} \Big\{ |\xi_j v^j | (1 - c_s^{ 2} ) + c_s W^{ - 1}
 \\
&\quad \times \sqrt {  (1-v^2c_s^2)(\xi_j\xi^j)-(1-c_s^2) (\xi_j v^j)^2  } \Big\}  + \frac{ | \xi_j \beta^j | }{\alpha}.
\end{align*}
\end{itemize}
\end{lemma}

The results in Lemmas \ref{lem:equDef}, \ref{lam:convexGgamma} and \ref{lam:propertyG} are consistent with
the special relativistic case established in \cite{WuTang2015,WuTang2016gEOS},
 if the spacetime is flat or $g_{\mu\nu}$ is the Minkowski metric $\mathrm{diag}\{-1,1,1,1\}$.

However, when $\bm  \Upsilon$ is not a constant matrix and changes in spacetime, the admissible state set $\mathcal G$ or ${\mathcal G}_\gamma$
 becomes {\em dependent on spacetime}. In other words, {\em the admissible state sets defined at different points in curved spacetime
are inequivalent, i.e. generally ${\mathcal G}_\gamma \neq {\mathcal G}_{\widehat{\gamma}}$ when $\bm  \Upsilon \neq \widehat{\bm  \Upsilon}$.} This makes it difficult to use the above properties of ${\mathcal G}_\gamma$
to develop PCP methods for GRHD equations \eqref{eq:GRHDcon}. The reason is
%The main difficulty lies in
that most existing techniques for designing bound-preserving type methods, see e.g. \cite{zhang2010b,zhang2012a,Hu2013,WuTang2015,QinShu2016,WuTang2016gEOS},  highly
depend on rewriting the target schemes into some forms of convex combination and then taking advantage of the convexity of the admissible state set.
{\em Whereas, unfortunately, in the present case the convexity does not hold between inequivalent admissible state sets
defined at different points in curved spacetime, making the related techniques invalidated.}

\subsection{Spacetime-independent admissible state set}

We find an effective solution to the above ``spacetime-dependent'' problem via a locally linear map.
Specifically, we map the admissible states defined at different points in curved spacetime into a common set
\begin{equation}\label{EQ-adm-set03}
\begin{split}
{\mathcal G}_* &= \bigg\{   {\bf W} = ({\tt W}_0,\cdots,{\tt W}_4)^\top \Big| {\tt W}_0>0,
\\
& \qquad \quad q( {\bf W}  ):= {\tt W}_4 - \bigg( \sum\limits_{i=0}^3 {\tt W}_i^2\bigg)^{\frac12}>0 \bigg\},
\end{split}
\end{equation}
in the sense of
\begin{equation}\label{eq:eqtoG1}
{\bf U}  \in {\mathcal G}_\gamma \quad \Longleftrightarrow \quad {\bf W}  := \sqrt{\gamma} \bm \Sigma {\bf U}  \in {\mathcal G}_*,
\end{equation}
where the  square  matrix $\bm \Sigma$ satisfies $\bm \Sigma^\top \bm \Sigma = {\rm diag}\{1,\bm \Upsilon ,1\}$.
One can take $\bm \Sigma$ as ${\rm diag}\{1,\bm \Upsilon ^{\frac12},1\}$, but a
better choice is explicitly defining $\bm \Sigma$ via the Cholesky decomposition of $\bm \Upsilon$ as follows
$$
\bm \Sigma =
\begin{bmatrix}
   1 & 0 & 0 & 0 & 0  \\
   0 & {\Sigma _{11} } & { \Sigma _{12} } & {\Sigma _{13} } & 0  \\
   0 & 0 & {\Sigma_{22} } & {\Sigma _{23} } & 0  \\
   0 & 0 & 0 & {\Sigma_{33} } & 0  \\
   0 & 0 & 0 & 0 & 1
\end{bmatrix},
$$
where
\begin{align*}
& \Sigma_{11} = \sqrt{\gamma^{11}},\quad \Sigma_{12} = \gamma^{12}/\sqrt{\gamma^{11}},\quad \Sigma_{13} = \gamma^{13}/\sqrt{\gamma^{11}},\\
& \Sigma_{22} = \sqrt{ \gamma^{22} - \big(\gamma^{12}\big)^2/\gamma^{11} },~  \Sigma_{23} =
\frac{1}{\Sigma_{22}}\bigg( \gamma^{23}- \frac{\gamma^{12} \gamma^{13}}{\gamma^{11}} \bigg),\\
& \Sigma_{33} = \sqrt{\gamma^{33}- \big({\gamma^{13}\big)^2}/{\gamma^{11}} -  \Sigma_{23}^2  }.
\end{align*}
It is worth noting that the transformation ${\bf U} \mapsto {\bf W}$ in \eqref{eq:eqtoG1} is linear in local spacetime.

The set ${\mathcal G}_*$ defined in \eqref{EQ-adm-set03} {\em does not depend on spacetime}.
In fact, ${\mathcal G}_*$ is equal to
the admissible state set in special relativistic case \cite{WuTang2015,WuTang2016gEOS}.
Hence it has the following properties, whose proofs are the same as the special RHD case in \cite{WuTang2015} and omitted here.

\begin{lemma}
\label{lam:convex}
The function $q({\bf W})$ defined in \eqref{EQ-adm-set03} is concave and Lipschitz continuous with respect to  $\bf W$. The admissible set ${\mathcal G}_*$ is an open convex set.
Moreover, $\lambda { {\bf W} }' + (1-\lambda) { {\bf W} }'' \in {\mathcal G}_* $ for any
${\bf W}' \in {\mathcal G}_*$, $ {\bf W}'' \in \overline{\mathcal G}_*$, and  $\lambda \in(0,1]$.
\end{lemma}

%\subsection{A new form of GRHD equations\\motivated by ${\mathcal G}_*$}

\subsection{${\mathcal G}_*$-associated formulation of GRHD equations}

The above analysis motivates us to develop PCP schemes for GRHD by taking advantages of the convexity of the
spacetime-independent set  ${\mathcal G}_*$.
Particularly, we would like to seek a new form of GRHD equations, whose admissible conservative vectors (state vectors)
exactly form the set ${\mathcal G}_*$.
To this end, we multiply Eqs. \eqref{eq:GRHDcon} by the invertible matrix $\bm  \Sigma$ from the left, and then
obtain the following equivalent form (abbreviated as ``W-form'' in later text)
\begin{equation}\label{eq:GRHDconWKL}
 \frac{\partial {\bf W} }{\partial t} +  \frac{\partial {\bf H}^j ({\bf W})}{\partial x^j}  =  {\bf S} (\bf W),
\end{equation}
where
\begin{align*}
&{\bf H}^j ({\bf W}) = \sqrt { - g} \bm  \Sigma {\bf F}^j,\\
& {\bf S} ( {\bf W} ) = \sqrt{\gamma} \frac{\partial \bm  \Sigma}{\partial t} {\bf U} +  \sqrt { - g} \bigg( {\bf Q} +\frac{\partial \bm  \Sigma}{\partial x^j} {\bf F}^j \bigg).
\end{align*}
For convenience, these notations omit the dependence of ${\bf H}^j$ and $\bf S$ on the metric $g_{\mu\nu}$ and its derivatives  $\frac{\partial g_{\mu\nu}}{\partial x^{\delta}}$.

Based on the relation \eqref{eq:eqtoG1},
the properties of ${\mathcal G}_\gamma$ established in Lemma \ref{lam:propertyG} can be directly
extended to ${\mathcal G}_*$.

\begin{lemma}\label{lam:LxF}
If ${\bf W} \in {\mathcal G}_*$, then
\begin{itemize}[\hspace{0em}$\bullet$]
  \item[{\rm(i). (Scaling invariance)}] $\lambda {\bf W} \in {\mathcal G}_* $ for any positive $\lambda$.
\item[{\rm(ii). (LxF splitting)}] for any vector $\bm  \xi =(\xi_1,\xi_2,\xi_3)\neq \bm 0$,
$$ {\bf W} \pm \eta_\xi^{-1}  {{ \xi_j  {\bf H}^j ( {\bf W} )}} \in \overline {\mathcal G}_*,$$
and
$$ {\bf W} \pm {\eta }^{-1} {{ \xi_j {\bf H}^j ( {\bf W} )}} \in {\mathcal G}_*,\quad \mbox{for any}~\eta > \eta_\xi,$$
where $\eta_\xi = \alpha \varrho_\xi$ is a bound of the spectral radius of the
Jacobian matrix $\partial (\xi_j  {\bf H}^j({\bf W})) / \partial {\bf W}$ with $\varrho_\xi$ defined in Lemma \ref{lam:propertyG}.
\end{itemize}
\end{lemma}% $\bm  U = {\gamma}^{-\frac12} \bm  \Sigma^{-1} \bm  W$,

The next two sections will utilize the theories established above to design the
provably PCP methods for the GRHD equations in W-form \eqref{eq:GRHDconWKL}.

\section{A First-order PCP scheme} \label{sec:FRscheme}

This section aims to establish the first theoretical result on PCP method for GRHD, i.e., rigorously show
the PCP property of the first-order Lax-Friedrichs (LxF) scheme for the GRHD equations in W-form \eqref{eq:GRHDconWKL} on a general mesh.
For convenience, we will also use $\bm x$ to denote $(x^1,x^2,x^3)$ in the following.

Assume that the three-dimensional ``spatial'' domain is divided into a mesh of cells $\{ {\mathcal I}_k\}$, such as tetrahedron or hexahedron elements.
For generality, the mesh can be unstructured. Let ${\mathcal N}_k$ denote the index set of all the neighboring cells of $ {\mathcal I}_k$.
For each $j \in {\mathcal N}_k$, let ${\mathcal E}_{kj}$ be the face of $ {\mathcal I}_k$ sharing with its neighboring cell $ {\mathcal I}_j$, i.e.
${\mathcal E}_{kj} =  \partial{\mathcal I}_k \cap \partial{\mathcal I}_j$,
and $\bm  {\xi}_{kj}= \big( {\xi}_{kj,1}, {\xi}_{kj,2}, {\xi}_{kj,3} \big)$ be the unit normal vector
of ${\mathcal E}_{kj}$ pointing from ${\mathcal I}_k$ to ${\mathcal I}_j$.
%The boundary (face) and volume of ${\mathcal C}_k$ are denoted by $\partial I_k$ and $|I_k|$ respectively. Let $\bm  n = (n^{(1)},n^{(2)},n^{(3)})$ denote the outward pointing unit normal vector of the boundary $\partial I_k$.
%Assume that each cell has $J$ faces.  Let $I_{k_j}$ denote the neighboring
%cell of $I_k$ along the $j$-th face of $I_k$, $j=1,2,\cdots,J$, and $\bm  n_j$ denote the outward pointing unit normal vector of the $j$-th face.
The time interval is also divided into mesh $\{t_0=0, t_{n+1}=t_n+\Delta t_{n}, n\geq 0\}$
with the time step-size $\Delta t_{n}$ determined by the CFL-type condition.

Integrating the W-form \eqref{eq:GRHDconWKL} over the cell ${\mathcal I}_k$ and using the
divergence theorem give
\begin{equation}\label{eq:int3D}
%\begin{split}
\frac{\rm d}{{\rm d} t} \int_{{\mathcal I}_k} {\bf W}  {\rm d} {\bm x}
+ \sum_{ j \in {\mathcal N}_k }
\int_{ {\mathcal E}_{kj} }    {\xi}_{kj,\ell} {\bf H}^{\ell}  {\rm d} S
=  \int_{{\mathcal I}_k} {\bf S}  {\rm d} {\bm x}.
%\end{split}
\end{equation}
Let $\overline{\bf W}_k^n$ be the approximation to the cell-average or the centroid-value of $\bf W$ over ${\mathcal I}_k$ at $t=t_n$.
Approximating the flux in \eqref{eq:int3D} by the LxF flux, and discretizing the time derivative by froward Euler method, one can derive
a first-order scheme
\begin{equation}\label{eq:1orderLF}
%\begin{split}
\overline{\bf W}_k^{n+1} =  \overline{\bf W}_k^{n}
 - \frac{\Delta t_n }{|{\mathcal I}_k|}  \sum_{ j \in {\mathcal N}_k }
 \big| {\mathcal E}_{kj} \big| \widehat{\bf H}_{kj}
+ \Delta t_n {\bf S} \big(  \overline{\bf W}_k^n \big),
%\end{split}
\end{equation}
where $\left| {\mathcal I}_k\right|$ and $\left|   {\mathcal E}_{kj} \right|$ respectively denote the volume of ${\mathcal I}_k$ and
the area of the face ${\mathcal E}_{kj}$.
%; $\Delta t_n>0$ is the time step-size determined by some appropriate CFL type condition.
The adopted LxF flux is
\begin{align*}
 \widehat{\bf H}_{kj}
 =  \frac{1}{2} {\xi}_{kj,\ell}
 \Big(   {\bf H}^{\ell} ( \overline{\bf W}_k^n ) + {\bf H}^{\ell} (\overline{\bf W}_{j}^n ) \Big)- \frac{ a_{kj} }{2} \big( \overline{\bf W}_{j}^n-  \overline{\bf W}_k^n \big),
\end{align*}
with the numerical viscosity coefficient satisfying
\begin{equation}\label{eq:defLFa}
a_{kj} \ge \max \big\{\eta_{ {\xi}_{kj} } ( \overline{\bf W}_k^n ), \eta_{{\xi}_{kj}} ( \overline{\bf W}_{j}^n ) \big\}.
\end{equation}
The readers are referred to
 Lemma \ref{lam:LxF} for the definition of $\eta_\xi$ for any nonzero vector $\bm  \xi \in {\mathbb{R}}^3$.
Here the corresponding cell-centered values of $g_{\mu\nu}$ are used to calculate ${\bf H}^{\ell} ( \overline{\bf W}_k^n )$ and ${\bf S} (  \overline{\bf W}_k^n )$.

Theorem \ref{thm:firstLF} shows that the scheme \eqref{eq:1orderLF} preserves $\overline{\bf W}_k^n \in {\mathcal G}_*$ under a CFL condition, see Appendix \ref{app:4} for its proof.

\begin{theorem}\label{thm:firstLF}
Assume that $\overline{\bf W}_k^0 \in {\mathcal G}_*$ for all $k$. Then the scheme \eqref{eq:1orderLF} is PCP
under the CFL-type condition
\begin{equation}\label{eq:cflLFfs}
 \Delta t_n \max \limits_{k}  \bigg(
\frac{1 }{2|{\mathcal I}_k|}  \sum_{ j \in {\mathcal N}_k }
a_{kj}
 \big| {\mathcal E}_{kj} \big|
 +\lambda_{\rm S} ( \overline{\bf W}_k^n) \bigg)< 1,
\end{equation}
where $\lambda_{\rm S} = 0$ if $q\big( {\bf S} ( \overline{\bf W}_k^n )\big) \ge 0 $, otherwise $\lambda_{\rm S}>0$ solves
%is the positive solution to
\begin{equation}\label{eq:LFs}
q\big( \overline{\bf W}_k^n + \lambda_{\rm S}^{-1}  {\bf S} ( \overline{\bf W}_k^n ) \big) = 0.
\end{equation}
\end{theorem}

\section{High-order PCP schemes} \label{sec:HIscheme}

This section is devoted to designing high-order, provably PCP schemes for the GRHD equations in W-form  \eqref{eq:GRHDconWKL}.
%, including the high-order accurate PCP finite difference, finite volume and discontinuous Galerkin (DG) schemes.

For the sake of convenience, we assume that the spatial domain is divided into a uniform cuboid mesh, with
the constant spatial step-size $\Delta_{\ell}$ in $x^\ell$-direction, $\ell=1,2,3$, respectively.
And the time interval is divided into mesh $\{t_0=0, t_{n+1}=t_n+\Delta t_{n}, n\geq 0\}$,
with the time step-size $\Delta t_{n}$ determined by the CFL-type condition.

To avoid confusing subscripts, in this section we sometimes use the symbol $\bm x$ or $({\tt x},{\tt y},{\tt z})$ to replace the independent
variables $(x^1,x^2,x^3)$.

\subsection{PCP finite volume and DG schemes} \label{sec:FVDG}

%\subsubsection{Setup}

Assume the uniform cuboid mesh is with cells
$$\big\{ {\mathcal I}_{ijk}=\big[ {\tt x}_{i-\frac{1}{2}}, {\tt x}_{i+\frac{1}{2}}\big]\times
\big[{\tt y}_{j-\frac{1}{2}},{\tt y}_{j+\frac{1}{2}}\big] \times \big[ {\tt z}_{k-\frac{1}{2}}, {\tt z}_{k+\frac{1}{2}}\big] \big\},$$
%The time interval is also divided into mesh $\{t_0=0, t_{n+1}=t_n+\Delta t_{n}, n\geq 0\}$
%with the time step-size $\Delta t_{n}$ determined by the CFL-type condition.
and $\overline {\bf W}_{ijk}^n $ be the numerical cell-averaged approximation of the exact solution ${\bf W}(t,{\bm x})$ over ${\mathcal I}_{ijk}$ at $t=t_n$.
We aim at designing PCP finite volume or DG type methods of the GRHD equations \eqref{eq:GRHDconWKL},
whose solution  $\overline {\bf W}_{ijk}^n$ always stays at ${\mathcal G}_*$ if $\overline {\bf W}_{ijk}^0\in {\mathcal G}_*$.

Towards achieving high-order ($({\tt K}+1)$-th order) spatial accuracy, the approximate solution polynomials ${\bf W}_{ijk}^n ({\bm x})$ of
degree $\tt K$ are also built usually,
to approximate the exact solution ${\bf W}(t_n,{\bm x})$ within the cell ${\mathcal I}_{ijk}$.
Such polynomial vector ${\bf W}_{ijk}^n ({\bm x})$ is, either reconstructed in finite volume methods
from $\big\{\overline {\bf W}_{ijk}^n \big\}$, or evolved in DG methods.
The cell-averaged value of ${\bf W}_{ijk}^n ({\bm x})$ over the cell ${\mathcal I}_{ijk}$ is required to be $\overline {\bf W}_{ijk}^{n}$.

\subsubsection{Method}

For the moment, the forward Euler method is used for time discretization, while high-order time discretization
will be considered later. Then, the main implementation procedures of our high-order (${\tt K} \ge 1$)
PCP finite volume or DG method can be outlined as follows.

\vspace{2mm}
\noindent
{\bf Step 0.} {Initialization.} Set $t=0$ and $n=0$, and compute $\overline {\bf W}_{ijk}^n$
 and $ {\bf W}_{ijk}^n ({\bm x}) $ for each cell ${\mathcal I}_{ijk}$
by using the initial data.
Note the convexity of ${\mathcal G}_*$ can ensure $\overline {\bf W}_{ijk}^n \in {\mathcal G}_*$.

%\vspace{2mm}
%\noindent
%{\bf Step 2.} {\bf Evolution.}

\vspace{2mm}
\noindent
{\bf Step 1.} Given admissible cell-averages $\big\{\overline {\bf W}_{ijk}^n\big\}$, perform PCP limiting procedure.
Use the {\tt PCP limiter} presented later to modify the polynomials $\big\{ {\bf W}_{ijk}^n ({\bm x}) \big\}$ as
$\big\{\widetilde {\bf W}_{ijk}^n ({\bm x}) \big\}$, such that the revised polynomials satisfy
\begin{equation}\label{eq:FVDGsuff}
\widetilde {\bf W}_{ijk}^n ({\bm x}) \in {\mathcal G}_*, \quad \mbox{for any}~{\bm x} \in {\mathbb{S}}_{ijk},
\end{equation}
where the set ${\mathbb{S}}_{ijk}$ consists of several important tensor-producted quadrature nodes in ${\mathcal I}_{ijk}$. Specifically,
\begin{equation*}%\label{eq:limiterset}
{\mathbb{S}}_{ijk} = \big( \widehat {\mathbb{S}}_i^{\tt x} \otimes {\mathbb{S}}_j^{\tt y} \otimes {\mathbb{S}}_k^{\tt z}     \big) \mathsmaller{\bigcup}
  \big( {\mathbb{S}}_i^{\tt x} \otimes \widehat {\mathbb{S}}_j^{\tt y} \otimes {\mathbb{S}}_k^{\tt z}     \big) \mathsmaller{\bigcup}
  \big( {\mathbb{S}}_i^{\tt x} \otimes {\mathbb{S}}_j^{\tt y} \otimes \widehat {\mathbb{S}}_k^{\tt z}     \big),
\end{equation*}
where $\widehat {\mathbb{S}}_i^{\tt x} = \{ \hat {\tt x}_i^{(\mu)}  \}_{\mu=1}^{\tt L}$,
$\widehat {\mathbb{S}}_j^{\tt y} = \{ \hat {\tt y}_j^{(\mu)}  \}_{\mu=1}^{\tt L}$,
$\widehat {\mathbb{S}}_k^{\tt z} = \{ \hat {\tt z}_k^{(\mu)}  \}_{\mu=1}^{\tt L}$ are
the $\tt L$-point Gauss-Lobatto quadrature nodes in the intervals $[{\tt x}_{i-\frac{1}{2}},{\tt x}_{i+\frac{1}{2}} ]$, $[{\tt y}_{j-\frac{1}{2}},{\tt y}_{j+\frac{1}{2}} ]$
and $[{\tt z}_{k-\frac{1}{2}},{\tt z}_{k+\frac{1}{2}} ]$, respectively; and
$ {\mathbb{S}}_i^{\tt x} = \{  {\tt x}_i^{(\mu)}  \}_{\mu=1}^{\tt Q}$,
${\mathbb{S}}_j^{\tt y} = \{  {\tt y}_j^{(\mu)}  \}_{\mu=1}^{\tt Q}$,
${\mathbb{S}}_k^{\tt z} = \{ {\tt z}_k^{(\mu)}  \}_{\mu=1}^{\tt Q}$
are the $\tt Q$-point Gauss-Legendre quadrature nodes in those three intervals respectively.

For achieving provably PCP property, $\tt L$ is suggested to satisfy $2{\tt L}-3\ge {\tt K}$.
For the accuracy requirement, $\tt Q$ shall satisfy:
$2{\tt Q} \ge {\tt K}+1$ for a $({\tt K}+1)$-th order finite volume method,
or ${\tt Q} \ge {\tt K}+1$ for a $\mathbb{P}^{\tt K}$-based DG method \cite{Cockburn0}.

\vspace{2mm}
\noindent
{\bf Step 2.} For each cell ${\mathcal I}_{ijk}$, evaluate the limiting values of $\widetilde {\bf W}_{ijk}^n ({\bm x})$
at the Gaussian points on the faces of the cell:
\begin{equation*}
\begin{split}
 & {\bf W}^{-,\mu,\nu}_{i+\frac{1}{2},j,k} \leftarrow  \widetilde{\bf W}_{ijk}^n \big( {\tt x}_{i+\frac12},{\tt y}_j^{(\mu)},{\tt z}_k^{(\nu)} \big),
\\ & {\bf W}^{+,\mu,\nu}_{i-\frac{1}{2},j,k} \leftarrow  \widetilde{\bf W}_{ijk}^n \big( {\tt x}_{i-\frac12},{\tt y}_j^{(\mu)},{\tt z}_k^{(\nu)} \big),
\\ & {\bf W}^{\mu,-,\nu}_{i,j+\frac{1}{2},k} \leftarrow  \widetilde{\bf W}_{ijk}^n \big( {\tt x}_i^{(\mu)},{\tt y}_{j+\frac12},{\tt z}_k^{(\nu)} \big),
\\ & {\bf W}^{\mu,+,\nu}_{i,j-\frac{1}{2},k} \leftarrow  \widetilde{\bf W}_{ijk}^n \big( {\tt x}_i^{(\mu)},{\tt y}_{j-\frac12},{\tt z}_k^{(\nu)} \big),
\\ & {\bf W}^{\mu,\nu,-}_{i,j,k+\frac{1}{2}} \leftarrow  \widetilde{\bf W}_{ijk}^n \big( {\tt x}_i^{(\mu)},{\tt y}_j^{(\nu)},{\tt z}_{k+\frac12} \big),
\\ & {\bf W}^{\mu,\nu,+}_{i,j,k-\frac{1}{2}} \leftarrow  \widetilde{\bf W}_{ijk}^n \big( {\tt x}_i^{(\mu)},{\tt y}_j^{(\nu)},{\tt z}_{k-\frac12} \big),
\end{split}
\end{equation*}
for $\mu,\nu=1,\cdots,{\tt Q}$.

\vspace{2mm}
\noindent
{\bf Step 3.} {Compute numerical fluxes.}
First estimate the upper bound $a^{(\ell)}_\star$ of the characteristic speed in $x^\ell$-direction by
\begin{equation}\label{eq:astarFVDG}
a^{(\ell)}_\star \ge \max\limits_{i,j,k,{\bm x}}  \left\{ \eta_{\xi_\ell} \big( \widetilde{\bf W}^n_{ijk} (\bm x) \big) \right\},
\end{equation}
with $\bm  \xi_\ell$ denoting the $\ell$-th row of unit matrix of size 3, $\ell=1,2,3$.
Let $\{\omega_\mu\}_{\mu=1}^{\tt Q}$ be the associated weights of the $\tt Q$-point Gauss-Legendre quadrature and
satisfy $\sum_{\mu=1}^{\tt Q}  \omega_\mu = 1$.
Then for each $i,j,k$, compute the numerical fluxes in $x^\ell$-direction, $\ell=1,2,3$, by
\begin{equation}\label{eq:FVDGnumflux}
\begin{split}
&
\widehat{\bf H}_{i+\frac12,j,k}^1 =    \omega_\mu \omega_\nu
 {\widehat {\bf H}}^1 \left( {\bf W}^{-,\mu,\nu}_{i+\frac{1}{2},j,k}, {\bf W}^{+,\mu,\nu}_{i+\frac{1}{2},j,k} \right),
\\
&
\widehat{\bf H}_{i,j+\frac12,k}^2 =  \omega_\mu \omega_\nu
 {\widehat {\bf H}}^2 \left( {\bf W}^{\mu,-,\nu}_{i,j+\frac{1}{2},k}, {\bf W}^{\mu,+,\nu}_{i,j+\frac{1}{2},k} \right),
\\
&
 \widehat{\bf H}_{i,j,k+\frac12}^3 = \omega_\mu \omega_\nu
 {\widehat {\bf H}}^3 \left( {\bf W}^{\mu,\nu,-}_{i,j,k+\frac{1}{2}}, {\bf W}^{\mu,\nu,+}_{i,j,k+\frac{1}{2}} \right),
\end{split}
\end{equation}
with summation convention employed, and
the numerical flux ${\widehat {\bf H}}^\ell \big({\bf W}^-,{\bf W}^+\big)$ taken as
the LxF flux
\begin{equation}\label{eq:WuLFflux}
\begin{split}
&{\widehat {\bf H}}^\ell \big({\bf W}^-,{\bf W}^+\big) =\frac12 \Big( \left( {\bf H}^\ell ( {\bf W}^- ) +  {\bf H}^\ell ( {\bf W}^+ ) \right)
\\
& \qquad  \qquad  \qquad \quad  -  a^{(\ell)}_\star ( {\bf W}^+ - {\bf W}^- ) \Big)  ,~~\ell =1,2,3.
\end{split}
\end{equation}
Numerical fluxes in \eqref{eq:FVDGnumflux} can be regarded as high-order
approximations to
\begin{align*}
&
 \frac{1}{\Delta_2 \Delta_3 }
\int_{ {\tt y}_{j-\frac{1}{2}}}^{{\tt y}_{j+\frac{1}{2}}} \int_{ {\tt z}_{k-\frac{1}{2}}}^{ {\tt z}_{k+\frac{1}{2}}}
   {\bf H} ^1 \big( t_n, {\tt x}_{i+\frac12}, {\tt y}, {\tt z} \big) {\rm d} {\tt y} {\rm d} {\tt z},
\\
&
 \frac{1}{\Delta_1 \Delta_3}
\int_{ {\tt x}_{i-\frac{1}{2}}}^{ {\tt x}_{i+\frac{1}{2}}} \int_{ {\tt z}_{k-\frac{1}{2}}}^{ {\tt z}_{k+\frac{1}{2}}}
   {\bf H} ^2 \big( t_n, {\tt x}, {\tt y}_{j+\frac12}, {\tt z} \big) {\rm d} {\tt x} {\rm d} {\tt z},
\\
&
 \frac{1}{\Delta_1 \Delta_2}
 \int_{ {\tt x}_{i-\frac{1}{2}}}^{ {\tt x}_{i+\frac{1}{2}}}
\int_{ {\tt y}_{j-\frac{1}{2}}}^{ {\tt y}_{j+\frac{1}{2}}}
   {\bf H} ^3 \big( t_n, {\tt x}, {\tt y}, {\tt z}_{k+\frac12} \big) {\rm d} {\tt x} {\rm d} {\tt y},
\end{align*}
respectively.

\vspace{2mm}
\noindent
{\bf Step 4.} Update the cell-averages by the scheme
\begin{equation} \label{eq:cellaverage}
\begin{split}
&\overline {\bf W}_{ijk}^{n+1} = \overline {\bf W}_{ijk}^{n}
- \frac{\Delta t_n} {\Delta_1} \left(  \widehat{\bf H}_{i+\frac12,j,k}^1 - \widehat{\bf H}_{i-\frac12,j,k}^1 \right)
\\
& \qquad \qquad
- \frac{\Delta t_n} {\Delta_2} \left( \widehat{\bf H}_{i,j+\frac12,k}^2 - \widehat{\bf H}_{i,j-\frac12,k}^2 \right)
\\
& \qquad  \quad
- \frac{\Delta t_n} {\Delta_3} \left(  \widehat{\bf H}_{i,j,k+\frac12}^3 -  \widehat{\bf H}_{i,j,k-\frac12}^3 \right)+ \Delta t_n  \overline{ {\bf S}}_{ijk}^n,
\end{split}
\end{equation}
where $\overline{ {\bf S}}_{ijk}^n$ denotes an appropriate high-order approximation to the cell-average of $\bf S$ over the cell ${\mathcal I}_{ijk}$, e.g.
$$
\overline{ {\bf S}}_{ijk}^n =  \omega_\delta \omega_\mu \omega_\nu
{\bf S} \left( \widetilde{\bf W}_{ijk}^n \big( {\tt x}_i^{(\delta)}, {\tt y}_j^{(\mu)}, {\tt z}_k^{(\nu)}  \big) \right),
$$
where Einstein's summation convention is used. Eq. \eqref{eq:cellaverage} is the formulation of
the finite volume scheme or the discrete equation for cell-averaged values in the DG scheme.
{\em As shown in Theorem \ref{thm:FVDGPCP} later, the PCP limiting procedure in {Step 1} can ensure the computed
$\overline {\bf W}_{ijk}^{n+1} \in {\mathcal G}$, which meets the condition of performing PCP limiting procedure
in the next time-forward step, see Step 6.}

\vspace{2mm}
\noindent
{\bf Step 5.} Built the polynomials $\big\{ {\bf W}_{ijk}^{n+1}({\bm x}) \big\}$. For a high-order finite volume scheme,
reconstruct the approximate solution polynomial ${\bf W}_{ijk}^{n+1}({\bm x})$ from the
cell averages $\big\{ \overline{\bf W}_{ijk}^{n+1} \big\}$; for $\mathbb{P}^{\tt K}$-based DG method $({\tt K}\ge 1)$,
evolve the high-order ``moments'' of  ${\bf W}_{ijk}^{n+1}({\bm x})$, similar to \eqref{eq:cellaverage}.
The details are omitted here, as these does not affect the PCP property of the proposed schemes.

%Update the solutions in the virtual cells with the given boundary conditions.

\vspace{2mm}
\noindent
{\bf Step 6.} Set $t_{n + 1}  = t_n  + \Delta t_n$. If $t_{n + 1}  < T_{\rm stop}$,
then assign $n \leftarrow n+1$ and go to {Step 1}, where the admissibility of $\big\{\overline {\bf W}_{ijk}^{n+1}\}$
has been ensured in Step 4. Otherwise, output numerical results and stop.

The main difference between the present PCP method and the traditional method is that
 the former adds a carefully designed PCP limiting procedure (i.e. Step 1).
 %Besides, the present discretization is conducted on
% the W-form \eqref{eq:GRHDconWKL} of GRHD equations, to achieve the provably PCP property.

\subsubsection{PCP limiter}

We now present the {\tt PCP limiter} used in {Step 1}, which is a key ingredient of the above high-order PCP method.
Without this limiter, the original high-order schemes are generally not PCP, and
may easily break down after some time steps in solving some ultra-relativistic problems involving low density or pressure, or very large velocity.
The notion of our {\tt PCP limiter} is extended from the non-relativistic case \cite{zhang2010b}
and special relativistic case \cite{QinShu2016,WuTang2016gEOS,WuTang2016}.
%It is a key ingredient of the above PCP methods.

%Its role is to detect and adaptively modify the polynomials $\big\{ {\bf W}_{ijk}^n ({\bf x}) \big\}$ as
%$\big\{\widetilde {\bf W}_{ijk}^n ({\bf x}) \big\}$, such that the revised/limited polynomials satisfy \eqref{eq:FVDGsuff}.
%This goal can be achieved by the {\tt PCP limiter}, only when the condition $\overline {\bf W}_{ijk}^n \in {\mathcal G}_*$ holds.
%Fortunately, this condition is preserved by the scheme \eqref{eq:cellaverage}, as it will be proved later.

To avoid the effect of the rounding error, we define
\begin{align} \label{set-G-epsilon}
{\mathcal G}_\epsilon = \left\{ { \left. {\bf W} = ( {\tt W}_0,\cdots, {\tt W}_4)^\top \right| {\tt W}_0\ge \epsilon,~q({\bf W})\ge \epsilon} \right\},
\end{align}
which is a subset of ${\mathcal G}_*$ and satisfy $\mathop {\lim }\limits_{\epsilon \to 0^+ } {\mathcal G}_\epsilon= {\mathcal G}_* $.
Here $\epsilon$ is a sufficiently small positive number and may be taken as $\epsilon=10^{-12}$ in numerical computations.

Under the condition $\overline {\bf W}_{ijk}^n \in {\mathcal G}_*$ in Step 1,
our PCP limiting procedure for each cell ${\mathcal I}_{ijk}$ is divided into the following easily-implemented steps.
For simplicity, here we temperately omit the superscripts $n$.%, which is fixed for given time-level $n$.

\begin{itemize}[\hspace{0em}$\bullet$]
  \item If $\overline {\bf W}_{ijk} \notin {\mathcal G}_\epsilon$, then the cell ${\mathcal I}_{ijk}$ is identified as vacuum region approximately.
  Set $\widetilde {\bf W}_{ijk} ({\bm x}) = \overline {\bf W}_{ijk}$ and skip the following steps.
  \item Enforce the first constraint in ${\mathcal G}_\epsilon$.
  Let
  ${\tt W}_{\ell,ijk}$ denote the $\ell$-th component of ${\bf W}_{ijk}$, and
  ${\tt W}_{0,\min} = {\min}_{ {\bm x} \in {\mathbb{S}}_{ijk}}  {\tt W}_{0,ijk} ( {\bm x} )$.
 If ${\tt W}_{0,\min} < \epsilon$,
then  $ {\tt W}_{0,ijk} ( {\bm x} )$ is limited as
$$
{\widehat {\tt W}}_{0,ijk} ( {\bm x} ) = \theta_1 \big(  {\tt W}_{0,ijk} ( {\bm x} ) - \overline {\tt W}_{0,ijk}  \big) + \overline {\tt W}_{0,ijk},
$$
where $\theta_1 = \big( \overline {\tt W}_{0,ijk} - \epsilon\big)/ \big( \overline {\tt W}_{0,ijk} - {\tt W}_{0,\min} \big) $. Otherwise,
take $ \widehat {\tt W}_{0,ijk} ( {\bf x} )=   {\tt W}_{0,ijk}( {\bm x} )$.
Denote
$$\quad \widehat {\bf W}_{ijk} ({\bm x}) :=  \big( \widehat {\tt W}_{0,ijk} ( {\bm x} ) ,  {\tt W}_{1,ijk} ( {\bm x} ),\cdots,  {\tt W}_{4,ijk} ( {\bm x} ) \big)^\top.$$
\item Enforce the second constraint in ${\mathcal G}_\epsilon$. Let
$q_{\min} = {\min}_{ {\bm x} \in {\mathbb S}_{ijk}}^{} q \big(\widehat {\bf W}_{ijk} ( {\bm x} ) \big)$. If $q_{\min} < \epsilon$,
then  $\widehat {\bf W}_{ijk} ( \bm x )$ is limited as
$$
\widetilde {\bf W}_{ijk} ({\bm x}) = \theta_2 \big( \widehat {\bf W}_{ijk} ({\bm x}) - \overline {\bf W}_{ijk} \big) + \overline {\bf W}_{ijk},
$$
where $\theta_2 = (q(\overline {\bf W}_{ijk}) - \epsilon)/ ( q(\overline {\bf W}_{ijk}) - q_{\min} ) $. Otherwise,
set $\widetilde {\bf W}_{ijk}  ({\bm x}) =  \widehat {\bf W}_{ijk} ({\bm x})$.
\end{itemize}

With the concavity of $q({\bf W})$, the above PCP limiting procedure yields
that the revised/limited polynomial $\widetilde {\bf W}_{ijk}^n (\bm x)$ satisfy \eqref{eq:FVDGsuff}.

%$\widetilde {\bf W}_{ijk}^n (\bf x)\in {\mathcal G}_\epsilon \subset {\mathcal G}_*$ for any ${\bf x}  \in {\mathbb S}_{ijk} $.
%Thus we immediately draw the following conclusion based on Theorem \ref{thm:FVDGPCP}.

In the end, we remark several features of the proposed {\tt PCP limiter}, in addition to its easy implementation.
The limiter keeps the conservativity, i.e.
\begin{align*}
 \frac{1}{\Delta_1 \Delta_2 \Delta_3 } \int_{I_{ijk}} \widetilde{\bf W}_{ijk}^n ( {\bm x} )  {\rm d} {\bm x} = \overline {\bf W}_{ijk}^n.
\end{align*}
It also maintains the high-order accuracy when ${\bf W}^n_{ijk}( {\bm x})$ approximates a smooth solution without vacuum, similar to \cite{zhang2010,zhang2010b}.
The above PCP limiting procedure is independently performed on each cell,
making the {\tt PCP limiter} easily parallel.

It is worth emphasizing that the {\tt PCP limiter} does not
depend on the reconstructing technique empolyed in Step 5 of a PCP finite volume scheme.
Theretofore, the proposed PCP finite volume schemes are very friendly,
in cooperation with any appropriate reconstructing techniques for ${\bf W}^n_{ijk}({\bm x})$,
e.g., essentially non-oscillatory (ENO) approach \cite{Hartten1987}, weighted ENO approach \cite{JiangShu1996},
piecewise parabolic method \cite{Colella1984}, etc.

\subsubsection{Provably PCP property}

We are now in position to present the theoretical result on the PCP property of the proposed finite volume and DG methods.

Before discussing high-order case $({\tt K} \ge 1)$, we first present the result for the special case of ${\tt K}=0$,
i.e., $ {\bf W}_{ijk}^n({\bm x})= \overline {{\bf W}}_{ijk}^{n}$. In this special case, the scheme \eqref{eq:cellaverage}
reduces to first-order LxF scheme, and the PCP limiting procedure is not required.
As a direct corollary of {Theorem \ref{thm:firstLF}}, we immediately have the following consequence.

\begin{corollary} \label{thm:FVDGk0}
When ${\tt K}=0$, %and $\overline {\bf W}_{ijk}^0\in {\mathcal G}_*$,
%then
the scheme \eqref{eq:cellaverage}
is PCP under the CFL-type condition
\begin{equation*}
\Delta t_n \left( \Delta_\ell^{-1} a^{(\ell)}_\star +  \lambda_{\rm S} \right) < 1,
\end{equation*}
with summation convention employed, where $\lambda_{\rm S} =0$ if $q\big( \overline{ {\bf S}}_{ijk}^n \big) \ge 0 $, otherwise $\lambda_{\rm S}>0$ is
the solution to
\begin{equation}\label{eq:FVDGs}
q\big( \overline{\bf W}_{ijk}^n + \lambda_{\rm S}^{-1} \overline{ {\bf S}}_{ijk}^n \big) = 0.
\end{equation}
\end{corollary}

Let $\{\hat \omega_\mu\}_{\mu=1}^{\tt L}$ be the associated weights of the $\tt L$-point Gauss-Lobatto quadrature, with
$\sum_{\mu=1}^{\tt L} \hat \omega_\mu = 1$ and $\hat \omega_1 = \hat \omega_{\tt L} = \frac{1}{{\tt L}({\tt L}-1)}$.
We can then rigorously show the PCP property of the proposed methods in high-order case ${\tt K} \ge 1$, as stated
in Theorem \ref{thm:FVDGPCP} with the proof displayed in Appendix \ref{app:5}.

\begin{theorem} \label{thm:FVDGPCP}
Assume ${\tt K} \ge 1$ and $\overline {\bf W}_{ijk}^0\in {\mathcal G}_*$ for all $i,j,k$.
Assume that the condition \eqref{eq:FVDGsuff} is satisfied by the revised polynomials $\big\{ \widetilde{\bf W}_{ijk}^n({\bm x}) \big\}$.
Then, under the CFL-type condition
\begin{equation}\label{eq:FVDGcfl}
\Delta t_n \left( \Delta_\ell^{-1} a^{(\ell)}_\star + \hat \omega_1 \lambda_{\rm S} \right) < \hat \omega_1,
\end{equation}
the scheme \eqref{eq:cellaverage} preserves $\overline {\bf W}_{ijk}^n\in {\mathcal G}_*$ for all $i,j,k,n$.
In other words, the scheme \eqref{eq:cellaverage} with ${\tt K} \ge 1$ is PCP under the condition \eqref{eq:FVDGcfl},
where $\lambda_{\rm S} =0$ if $q\big( \overline{ {\bf S}}_{ijk}^n \big) \ge 0 $, otherwise $\lambda_{\rm S}$ is
the positive solution of Eq. \eqref{eq:FVDGs}.
\end{theorem}

For some high-order finite volume methods, it only needs to reconstruct
the limiting values $\left\{{\bf W}^{\pm,\mu,\nu}_{i\mp\frac{1}{2},i,j},
{\bf W}^{\mu,\pm,\nu}_{i,j\mp\frac{1}{2},k}, {\bf W}^{\mu,\nu,\pm}_{i,j,k\mp\frac{1}{2}}\right\}$
instead of the polynomial ${\bf W}^n_{ijk}({\bm x})$. In this case,  based on the proof of Theorem \ref{thm:FVDGPCP},
the condition \eqref{eq:FVDGsuff} for achieving PCP property can
be replaced with the following condition
{\small
\begin{equation*}
\begin{split}
&
{\bf W}^{\pm,\mu,\nu}_{i\mp\frac{1}{2},j,k},~
{\bf W}^{\mu,\pm,\nu}_{i,j\mp\frac{1}{2},k},~
{\bf W}^{\mu,\nu,\pm}_{i,j,k\mp\frac{1}{2}} \in {\mathcal G}_*,\quad \mu,\nu=1,\cdots,{\tt Q},
\\
&
\frac{1}{1-2\hat \omega_1} \bigg\{
 \overline{\bf W}_{ijk}^n   -
  \frac{ \hat \omega_1 \omega_\mu \omega_\nu  }{  \Delta_\ell^{-1}  a_\star^{(\ell)} }
  \Big(   \Delta_1^{-1}  a_\star^{(1)} \big(  {\bf W}_{i-\frac12,j,k}^{+,\mu,\nu}   +  {\bf W}_{i+\frac12,j,k}^{-,\mu,\nu} \big)
\\
& \qquad \qquad \quad
+ \Delta_2^{-1}  a_\star^{(2)} \big(  {\bf W}_{i,j-\frac12,k}^{\mu,+,\nu} +  {\bf W}_{i,j+\frac12,k}^{\mu,-,\nu}   \big)
\\
&\qquad  \qquad
+ \Delta_3^{-1}  a_\star^{(3)} \big(  {\bf W}_{i,j,k-\frac12}^{\mu,\nu,+} +
{\bf W}_{i,j,k+\frac12}^{\mu,\nu,-} \big) \Big) \bigg\} = \Pi_\star \in {\mathcal G}_*,
\end{split}
\end{equation*}
}for all $i,j,k$, where $\Pi_\star$ is defined in the proof of Theorem \ref{thm:FVDGPCP}.
Similar to the discussions in Sec. 5 of \cite{zhang2011b},
the previous PCP limiting procedure can be easily revised to meet such condition.

\subsubsection{Remarks}

%Besides, for strong discontinuity problem with high-order PCP DG method,
%nonlinear limiter such as TVB, WENO, etc. shall be used at the end of this step, to enhance nonlinear stability.

%source term not affect

The scheme \eqref{eq:cellaverage}
is only first-order accurate in time. To achieve high-order
 PCP scheme in time, one can replace the forward Euler time discretization in \eqref{eq:cellaverage}
 with high-order strong stability preserving (SSP)  methods \cite{Gottlieb2009}.

For example, utilizing the third-order SSP Runge-Kutta method gives
 \begin{equation} \label{eq:RK1}
 {\small
 \begin{aligned}
 & \overline {\bf W}^ *_{ijk}   = \overline {\bf W}^n_{ijk}  + \Delta t_n  {\bf L}_{ijk} \big(  \widetilde {\bf W}^n ( {\bm x} ) \big), \\
 & \overline {\bf W}^{ *  * }_{ijk}  = \frac{3}{4} \overline {\bf W}^n_{ijk}  + \frac{1}{4}\left( \overline {\bf W}^ *_{ijk}
  + \Delta t_n {\bf L}_{ijk} \big(  \widetilde {\bf W}^{*} ( {\bm x} ) \big) \right), \\
 & \overline {\bf W}^{n+1}_{ijk}  = \frac{1}{3} \overline {\bf W}^n_{ijk}  + \frac{2}{3}\left( \overline {\bf W}^{ *  * }_{ijk}
 + \Delta t_n  {\bf{L}}_{ijk} \big(  \widetilde {\bf W}^{**} ( {\bm x} )  \big) \right),
 \end{aligned}}
 \end{equation}
 where ${\bf L}_{ijk} \big( \widetilde{\bf W} ( {\bm x} )\big)$ is the numerical spatial operator,
 and $\widetilde {\bf W}^n({\bm x})$, $\widetilde {\bf W}^ *({\bm x})$, $\widetilde {\bf W}^ {**}({\bm x})$
 denote the PCP limited versions of the reconstructed or evolved polynomial vector at each Runge-Kutta stage.
 Since such SSP method is a convex combination of the forward Euler method,
according to the convexity of ${\mathcal G}_*$,
 the resulting high-order scheme \eqref{eq:RK1} is also
 PCP under the CFL condition \eqref{eq:FVDGcfl}.

To enforce the condition \eqref{eq:FVDGcfl} rigorously, we need to get an
accurate estimation of $a^{(\ell)}_\star$  for all the Rung-Kutta stages in \eqref{eq:RK1} based only on the numerical solution at time level $n$, which
is highly nontrivial. Hence, in practical computations, we suggest to take the value of $a^{(\ell)}_\star$ slightly larger.
Besides, the time step-size selecting strategy suggested in \cite{WangZhang2012} may be adopted to improve computational efficiency.

 The high-order SSP multi-step method can also be used for time discretization to achieve
 high-order PCP schemes c.f. \cite{WuTang2016gEOS}, and the details are omitted here.
The above complication of enforcing the condition \eqref{eq:FVDGcfl} does not exist if one uses a SSP multi-step time discretization.

\subsection{PCP finite difference scheme} \label{sec:FD}

Assume the uniform cuboid mesh with grid points $\{ ( {\tt x}_i, {\tt y}_j, {\tt z}_k ) \}$,
%cells $\{({\tt x}_i,x_{i+1})\times({\tt y}_{j},{\tt y}_{j+1}) \times ({\tt z}_{k},{\tt z}_{k+1}) \},$
%and the spatial step-size denoted by $\Delta_\ell$ in $x^\ell$-direction, $\ell=1,2,3$. The time interval is divided into a mesh $\{t_0=0, t_{n+1}=t_n+\Delta t_{n}, n\geq 0\}$
%with the time step-size $\Delta t_{n}$ determined by the CFL-type condition.
%Let
and $ {\bf W}_{i,j,k}^n $ denote the numerical approximation to the value of the exact solution  ${\bf W}(t_n,{\tt x}_i,{\tt y}_j,{\tt z}_k)$ at the grid point.
We would like to design PCP finite difference schemes of the GRHD equations in W-form \eqref{eq:GRHDconWKL},
which preserve ${\bf W}_{i,j,k}^n \in {\mathcal G}_*$ if ${\bf W}_{i,j,k}^0\in {\mathcal G}_*$.

\subsubsection{Method}

We also focus on the forward Euler time discretization first, and consider high-order time discretization later.
Then, a $\tt r$-th order (spatially) accurate, conservative finite difference scheme of the GRHD equations \eqref{eq:GRHDconWKL}
may be written as
\begin{equation}\label{eq:FD}
{\bf W}_{i,j,k}^{n+1} = {\bf W}_{i,j,k}^n + \Delta t_n {\bf{L}}_{ijk} ( {\bf W}^n),
\end{equation}
with
\begin{equation} \label{eq:FDLoperator}
{\small
\begin{aligned}
& {\bf{L}}_{ijk} \left( {\bf W}^n\right) := {\bf S} \big( {\bf W}_{i,j,k}^n \big)
+ \frac {\widehat{\bf H}^1_{i-\frac12,j,k} - \widehat{\bf H}^1_{i+\frac12,j,k}} {\Delta_1}
\\
&\quad + \frac{\widehat{\bf H}^2_{i,j-\frac12,k} - \widehat{\bf H}^2_{i,j+\frac12,k} }{\Delta_2}
+ \frac{\widehat{\bf H}^3_{i,j,k-\frac12} - \widehat{\bf H}^3_{i,j,k+\frac12}}{\Delta_3} .
\end{aligned}}
\end{equation}
Here $ \widehat{\bf H}^1_{i+\frac12,j,k},\widehat{\bf H}^2_{i,j+\frac12,k}$ and $\widehat{\bf H}^3_{i,j,k+\frac12}$
are the numerical fluxes consistent with the fluxes ${\bf H}^1 ({\bf W})$, ${\bf H}^2 ({\bf W})$ and ${\bf H}^3 ({\bf W})$ respectively,
satisfying that the last three terms in \eqref{eq:FDLoperator} are respectively $\tt r$-th order approximations to
$$-\frac{\partial {\bf H}^\ell (\bf W)}{\partial x^\ell} (t_n,{\tt x}_i,{\tt y}_j,{\tt z}_k), \quad \ell=1,2,3.$$

There are lots of approaches, e.g. \cite{JiangShu1996,SureshHuynh1997,BalsaraShu2000}, to
get high-order $({\tt r} > 1)$ numerical fluxes. However, the resulting high-order schemes
are generally not PCP, and may easily break down when solving some demanding extreme problems due to the nonphysical numerical solutions
$ {\bf W}_{i,j,k}^n \notin {\mathcal G}_*$.

In order to preserve $ {\bf W}_{i,j,k}^n \in {\mathcal G}_*$, the numerical fluxes in
our high-order PCP finite difference method are carefully designed with a PCP flux limiter.
The outline of the implementing procedures are as follows.

\vspace{2mm}
\noindent
{\bf Step 0.} Initialization. Set $t=0$ and $n=0$, and use the initial data to assign the value of ${\bf W}_{i,j,k}^0$
at each grid point. Physically, ${\bf W}_{i,j,k}^0 \in {\mathcal G}_*$.

%\vspace{2mm},
%\noindent
%{\bf Step 2.} {\bf Evolution.}

\vspace{2mm}
\noindent
{\bf Step 1.} Compute high-order flux. Use an appropriate traditional technique, e.g. ENO \cite{Hartten1987}, WENO \cite{JiangShu1996},
monotonicity-preserving approaches \cite{SureshHuynh1997,BalsaraShu2000} etc., to
construct the $\tt r$-th order numerical fluxes, denoted by
$ \widehat{\bf H}^{1,{\tt high}}_{i+\frac12,j,k}$, $\widehat{\bf H}^{2,{\tt high}}_{i,j+\frac12,k}$ and $\widehat{\bf H}^{3,{\tt high}}_{i,j,k+\frac12}$.

\vspace{2mm}
\noindent
{\bf Step 2.} Compute first-order PCP flux. Since the LxF flux is shown to be PCP in Theorem \ref{thm:firstLF}, we compute
\begin{equation*}%\label{eq:LFflux}
\begin{split}
&
\widehat{\bf H}^{1,{\tt LF}}_{i+\frac12,j,k} = {\widehat {\bf H}}^1 \big({\bf W}^n_{i,j,k},{\bf W}^n_{i+1,j,k}\big),\\
&
\widehat{\bf H}^{2,{\tt LF}}_{i,j+\frac12,k} = {\widehat {\bf H}}^2 \big( {\bf W}^n_{i,j,k},{\bf W}^n_{i,j+1,k} \big),\\
&
\widehat{\bf H}^{3,{\tt LF}}_{i,j,k+\frac12} = {\widehat {\bf H}}^3 \big( {\bf W}^n_{i,j,k},{\bf W}^n_{i,j,k+1} \big),
\end{split}
\end{equation*}
where ${\widehat {\bf H}}^\ell \big({\bf W}^-,{\bf W}^+\big)$ is the LxF flux defined in \eqref{eq:WuLFflux}, with
\begin{equation}\label{eq:astarFD}
a_\star^{(\ell)} \ge \max\limits_{i,j,k} \eta_{\xi_\ell} ( {\bf W}^n_{i,j,k} ),
\end{equation}
and $\bm \xi_\ell$ denoting the $\ell$-th row of the unit matrix of size 3.

\vspace{2mm}
\noindent
{\bf Step 3.} Limit high-order flux. Modify the high-order fluxes
$ \widehat{\bf H}^{1,{\tt high}}_{i+\frac12,j,k}$, $\widehat{\bf H}^{2,{\tt high}}_{i,j+\frac12,k}$ and $\widehat{\bf H}^{3,{\tt high}}_{i,j,k+\frac12}$
to high-order PCP fluxes defined by
{\small
\begin{equation}\label{eq:PCPflux}
\begin{split}
& \widehat{\bf H}^{1}_{i+\frac12,j,k} = \theta_{i+\frac12,j,k} \Big( \widehat{\bf H}^{1,{\tt high}}_{i+\frac12,j,k}
- \widehat{\bf  H}^{1,{\tt LF}}_{i+\frac12,j,k} \Big) + \widehat{\bf H}^{1,{\tt LF}}_{i+\frac12,j,k} ,
\\[1mm]
&
\widehat{\bf H}^{2}_{i,j+\frac12,k} = \theta_{i,j+\frac12,k} \Big( \widehat{\bf H}^{2,{\tt high}}_{i,j+\frac12,k}
- \widehat{\bf H}^{2,{\tt LF}}_{i,j+\frac12,k} \Big) + \widehat{\bf H}^{2,{\tt LF}}_{i,j+\frac12,k},
\\[1mm]
&
\widehat{\bf H}^{3}_{i,j,k+\frac12} =  \theta_{i,j,k+\frac12} \Big( \widehat{\bf H}^{3,{\tt high}}_{i,j,k+\frac12}
-\widehat{\bf H}^{3,{\tt LF}}_{i,j,k+\frac12}  \Big) + \widehat{\bf H}^{3,{\tt LF}}_{i,j,k+\frac12},
\end{split}
\end{equation}
}via the {\tt PCP flux limiter} presented later. In \eqref{eq:PCPflux}, the parameters
$\theta_{i+\frac12,j,k},\theta_{i,j+\frac12,k},\theta_{i,j,k+\frac12} \in [0,1]$, whose computation
is the main ingredient of {\tt PCP flux limiter}.

\vspace{2mm}
\noindent
{\bf Step 4.}  Evolve forward by the scheme \eqref{eq:FD}--\eqref{eq:FDLoperator} with the high-order PCP fluxes in \eqref{eq:PCPflux}.

\vspace{2mm}
\noindent
{\bf Step 5.}
Set $t_{n + 1}  = t_n  + \Delta t_n$. If $t_{n + 1}  < T_{\rm stop}$,
then assign $n \leftarrow n+1$ and go to {Step 1}. Otherwise, output numerical results and stop.

%Justification and

\subsubsection{PCP flux limiter}

The {\tt PCP flux limiter} used in {Step 3} is the key point in designing the above PCP finite difference scheme.
Its role is to locally modify any appropriate high-order numerical fluxes into high-order PCP fluxes of form \eqref{eq:PCPflux}.

For the sake of convenience, the following notations are introduced. We employ the vector $\bm \theta_{ijk}$
to represent the parameters
\begin{equation*}
\Big(\theta_{i-\frac12,j,k},\theta_{i+\frac12,j,k},\theta_{i,j-\frac12,k},\theta_{i,j+\frac12,k},\theta_{i,j,k-\frac12},\theta_{i,j,k+\frac12} \Big),
\end{equation*}
and the $\ell$-th component of $\bm \theta_{ijk}$ is also denoted by $\theta_{ijk}^{(\ell)}$, $\ell=1,\cdots,6$.
We also use the notation
\begin{equation*}%\label{eq:Wtheta}
{\bf W}_{i,j,k}( \bm \theta_{ijk} ) :=  {\bf W}_{i,j,k}^n +
\Delta t_n {\bf L}_{ijk} ( {\bf W}^n  ) ,
\end{equation*}
to explicitly display the dependance of ${\bf W}_{i,j,k}^{n+1}$ on $\bm \theta_{ijk}$. Then, ${\bf W}_{i,j,k}( \bm \theta_{ijk} )$ can
be reformulated as
\begin{equation}\label{eq:defWtheta}
{\bf W}_{i,j,k}( \bm \theta_{ijk} ) = {\bf W}_{i,j,k}( {\bf 0} ) + \sum_{\ell=1}^6 \theta_{ijk}^{(\ell)} {\bf C}_\ell,
\end{equation}
with
\begin{equation}\label{eq:defC}
\begin{split}
{\bf C}_{1,2} :=  \pm \frac{\Delta t_n}{\Delta_1} \Big( \widehat{\bf H}^{1,{\tt high}}_{i\mp \frac12,j,k}
- \widehat{\bf  H}^{1,{\tt LF}}_{i\mp \frac12,j,k}  \Big),
\\
{\bf C}_{3,4} :=  \pm \frac{\Delta t_n}{\Delta_2}
\Big( \widehat{\bf H}^{2,{\tt high}}_{i,j\mp\frac12,k}
- \widehat{\bf H}^{2,{\tt LF}}_{i,j\mp\frac12,k} \Big),
\\
{\bf C}_{5,6} :=  \pm \frac{\Delta t_n}{\Delta_3}
 \Big( \widehat{\bf H}^{3,{\tt high}}_{i,j,k\mp\frac12}
-\widehat{\bf H}^{3,{\tt LF}}_{i,j,k\mp\frac12}  \Big).
\end{split}
\end{equation}

{\em Our goal is to carefully choose the parameters $\bm \theta_{ijk}$ such that ${\bf W}_{i,j,k}( \bm \theta_{ijk} ) ={\bf W}_{i,j,k}^{n+1} \in {\mathcal G}_*$
provided ${\bf W}_{i,j,k}^{n} \in {\mathcal G}_*$.}

Simply taking $\bm \theta_{ijk}={\bf 0}$ in \eqref{eq:PCPflux} gives a PCP scheme, which is exactly the LxF scheme. And
the following corollary directly follows from Theorem \ref{thm:firstLF}.

\begin{corollary} \label{thm:FDr1}
%If taking $\bm \theta_{ijk}={\bf 0}$ in \eqref{eq:PCPflux}, then
%the scheme \eqref{eq:FD} is PCP under the CFL-type condition
If ${\bf W}_{i,j,k}^{n} \in {\mathcal G}_*$, then ${\bf W}_{i,j,k}( \bm 0 )\in {\mathcal G}_*$ under the CFL-type condition
\begin{equation}\label{eq:FDLFcfl}
\Delta t_n \left( \Delta_\ell^{-1} a^{(\ell)}_\star +  \lambda_{\rm S} \right) < 1,
\end{equation}
where $\lambda_{\rm S} =0$ if $q\big( {\bf S} ( {\bf W}_{i,j,k}^n )\big) \ge 0 $, otherwise $\lambda_{\rm S}>0$ and solves
$
q\big( {\bf W}_{ijk}^n + \lambda_{\rm S}^{-1} {\bf S}  ( {\bf W}_{i,j,k}^n ) \big) = 0.
$
\end{corollary}

However, such an approach (taking $\bm \theta_{ijk}={\bf 0}$) evidently destroys the orignal high-order accuracy and deprives
 the significance of constructing high-order numerical flux in  {Step 1}.
In order to maintain the $\tt r$-th order accuracy of the original numerical flux, each component of the
parameters $\bm \theta_{ijk}$ is expected to be $1-{\mathcal O( {\max}_\ell \{\Delta_\ell \}^{\tt r})}$ for smooth solutions.

%To achieve the PCP property,
There exist in the literature two types of positivity-preserving flux limiters, which can be borrowed and extended to the GRHD case, including
 the cut-off flux limiter \cite{Hu2013} and the parametrized flux limiter \cite{Xu_MC2013,Liang2014,JiangXu2013,XiongQiuXu2014,Christlieb2016}.
The extension of the cut-off limiter to the GRHD case is similar to the special RHD case \cite{WuTang2015}.
In the following, we mainly focus on developing parametrized PCP flux limiter, because the parametrized limiter works well
in maintaining the high-order accuracy \cite{Xu_MC2013}.

The parametrized PCP flux limiter attempts to seek the almost ``best'' parameters ${\bm \theta}_{ijk}$, such that
each parameter is as close to 1 as possible while subject to ${\bf W}_{i,j,k}( \bm \theta_{ijk} ) \in {\mathcal G}_*$.
More specifically, such ${\bm \theta}_{ijk}$ can be computed through the following two sub-steps of {Step 3}.

\begin{enumerate}[\hspace{0em}$\bullet$]
  \item[{\bf Step 3.1.}]
  For each $i,j,k$, find large parameters $\Lambda_{i,j,k}^{(\ell)}\in [0,1],\ell=1,\cdots,6$,
  such that
  $$
   {\bf W}_{i,j,k}({\bm \theta}) \in {\mathcal G}_*,~~\mbox{for all}~{\bm \theta} \in \Theta^\star_{i,j,k} := \bigotimes_{\ell=1}^6 \big[0,\Lambda_{i,j,k}^{(\ell)}\big] ,
  $$
  where the symbol ``$\otimes$'' denotes tensor product.
  \item[{\bf Step 3.2.}] For each $i,j,k$, set
  \begin{equation*}
    \begin{split}
    &\theta_{i+\frac12,j,k}=\min \left\{ \Lambda_{i,j,k}^{(2)}, \Lambda_{i+1,j,k}^{(1)} \right\},\\
    &\theta_{i,j+\frac12,k}=\min \left\{ \Lambda_{i,j,k}^{(4)}, \Lambda_{i,j+1,k}^{(3)} \right\},\\
    &\theta_{i,j,k+\frac12}=\min \left\{ \Lambda_{i,j,k}^{(6)}, \Lambda_{i,j,k+1}^{(5)} \right\}.
    \end{split}
  \end{equation*}
\end{enumerate}

In following, we shall present the details of {Step 3.1}.
Specifically, we need to determine the hyperrectangular $\Theta^\star_{i,j,k}$ for given values of ${\bf W}_{i,j,k}( {\bf 0} )$
and $\{ {\bf C}_\ell \}_{\ell = 1}^6$ defined in \eqref{eq:defC}.

To avoid the effect of the rounding error, we introduce a small positive number
$$\epsilon ={\min}\left\{ 10^{-12}, \mathop{\min}\limits_{i,j,k} \{ {\tt W}_{0,i,j,k}( {\bf 0} ) \},
\mathop{\min}\limits_{i,j,k} \big\{ q \big( {\bf W}_{i,j,k}( {\bf 0} ) \big) \big\} \right\},$$
where ${\tt W}_{0,i,j,k}({\bm \theta})$ denotes the first component of  ${\bf W}_{i,j,k}({\bm \theta})$.
Under the condition \eqref{eq:FDLFcfl},
${\bf W}_{i,j,k}( {\bf 0} ) \in {\mathcal G}_*$ implies $\epsilon>0$, and
that ${\bf W}_{i,j,k}( {\bf 0} )$ belongs to $ {\mathcal G}_\epsilon$ defined in \eqref{set-G-epsilon}. We then have the following property, whose proof is displayed in Appendix \ref{app:6}.

\begin{lemma}\label{lem:twosets}
Under the condition \eqref{eq:FDLFcfl}, the two sets
\begin{equation*}
\begin{split}
& \Theta_0 = \big\{ {\bm \theta} \in [0,1]^6~\big|~
{\tt W}_{0,i,j,k}( {\bm \theta} ) \ge \epsilon  \big\},\\
& \Theta = \big\{ {\bm \theta} \in [0,1]^6~\big|~
{\tt W}_{0,i,j,k}( {\bm \theta} ) \ge \epsilon,~q\big( {\bf W}_{i,j,k} ( {\bm \theta} ) \big)  \ge \epsilon\big\},
\end{split}
\end{equation*}
are both convex.
\end{lemma}

Based on this lemma, Step 3.1 is divided into the following two sub-steps for each $i,j,k$.

\begin{enumerate}[\hspace{0em}$\bullet$]
  \item[{\bf Step 3.1(a)}]
  Find a big hyperrectangular
  $$
  \Theta_0^\star: =  \bigotimes_{\ell=1}^6 \big[0,\Lambda_0^{(\ell)}\big]  \subseteq \Theta_0.
  $$
  within the convex set $\Theta_0$, and $\Lambda_{0}^{(\ell)} $, $\ell=1,\cdots,6$, should be as large as possible.
  Specifically, they are computed by
  \begin{equation}
  \Lambda_{0}^{(\ell)} =
 \begin{cases}
 \min\left\{ 1, \frac{ {\tt W}_{0,i,j,k} ({\bf 0}) -\epsilon} {\varepsilon_0 +
  \sum\limits_{   \mu \in {\mathcal N}_c } \left| {\rm C}_{0,\mu} \right| }  \right\}, &  \text{if ${\rm C}_{0,\ell} < 0$,} \\
    1, &  \text{otherwise,}
 \end{cases}                \end{equation}
  where ${\rm C}_{0,\ell}$ denotes the first component of ${\bf C}_{\ell}$,
  the index set ${\mathcal N}_c = \{ \mu \in \{1,2,\cdots,6\}~|~{\rm C}_{0,\mu} <0 \}$,
  and $\varepsilon_0>0$ is a small parameter to avoid division by zero and may be taken as $10^{-12}$.
  \item[{\bf Step 3.1(b)}]
  Shrink the hyperrectangular $\Theta_0^\star$ into $\Theta^\star_{i,j,k}$ such that $\Theta^\star_{i,j,k} \subseteq \Theta$. Let
  $$
   \bm {\mathcal {V}}_{\ell_1,\ell_2,\ell_3,\ell_4,\ell_5,\ell_6  }  = \Big( \ell_1 \Lambda_{0}^{(1)} ,\ell_2 \Lambda_{0}^{(2)},\cdots,\ell_6 \Lambda_{0}^{(6)} \Big),
  $$
  with $\ell_1,\cdots,\ell_6 \in \{0,1\}$,
  denote 64 vertices of the hyperrectangular $\Theta_0^\star$.
  For each vertex $\bm {\mathcal {V}}_{\ell_1,\ell_2,\ell_3,\ell_4,\ell_5,\ell_6  }=: \bm {\mathcal {V}}_{\bm \ell} $ do the following:
  \begin{enumerate}[\hspace{0em}$\bullet$]
  \item If $ \bm {\mathcal {V}}_{\bm \ell}  \in \Theta$, then set $\bm {\hat{\mathcal {V}}}_{ \bm \ell }= \bm {\mathcal {V}}_{ \bm \ell }$;
  \item Otherwise, compute the unique solution of
  the equation $q\big( \bm W_{i,j,k} (\lambda \bm {\mathcal {V}}_{\bm \ell }  )  \big) = \epsilon$ for the unknown $\lambda\in[0,1)$,
  and set $\bm {\hat{\mathcal {V}}}_{ \bm \ell }= \lambda \bm {\mathcal {V}}_{ \bm \ell }$. Here the uniqueness of $\lambda$ is ensured by
  the concavity of the function $q({\bf W})$ stated in Lemma \ref{lam:convex}.
  \end{enumerate}
  This gives
  $\bm {\hat{\mathcal {V}}}_{ \bm \ell } =: \big(  \hat{\mathcal {V}}_{\bm \ell}^{(1)}, \hat{\mathcal {V}}_{\bm \ell}^{(2)} ,\cdots, \hat{\mathcal {V}}_{\bm \ell}^{(6)} \big) \in \Theta$.
  Finally, compute the parameters $\big\{\Lambda_{i,j,k}^{(\mu)} \big\}_{\mu=1}^6$ by
  \begin{align*}
   \Lambda_{i,j,k}^{(\mu)}  = \min\limits_{ \ell_1,\cdots,\ell_6 \in \{0,1\}~{\rm and }~\ell_\mu = 1 }  {\hat{\mathcal {V}}}_{\bm \ell}^{(\mu)},
  \end{align*}
which determines the hyperrectangular
$$\Theta^\star_{i,j,k} = \bigotimes_{\mu=1}^6 \big[0,\Lambda_{i,j,k}^{(\mu)}\big].$$
\end{enumerate}

\subsubsection{Provably PCP property}

We now study in theory the PCP property of the above high-order finite difference scheme.

Based on the computing approach of the parameters $\big\{\Lambda_{i,j,k}^{(\ell)} \big\}_{\ell=1}^6$ displayed in {Step 3.1(a)} and {Step 3.1(b)},
one has $\Theta^\star_{i,j,k} \subset \Theta$. This implies
$${\bf W}_{i,j,k}({\bm \theta}) \in {\mathcal G}_\epsilon,~\mbox{for all}~{\bm \theta} \in \Theta^\star_{i,j,k}.$$
From the definition of $\bm \theta_{ijk}$ in {Step 3.2}, we obtain
$$
0 \le \theta_{ijk}^{(\ell)} \le \Lambda_{i,j,k}^{(\ell)},\quad \ell=1,\cdots,6,
$$
Thus ${\bm \theta}_{ijk} \in \Theta^\star_{i,j,k}$, and ${\bf W}^{n+1}_{i,j,k} = {\bf W}_{i,j,k}({\bm \theta}_{ijk}) \in {\mathcal G}_\epsilon$.
We then immediately draw the following conclusion.

\begin{theorem}
Assume that the numerical fluxes in \eqref{eq:FDLoperator} are taken as the high-order PCP fluxes in \eqref{eq:PCPflux}, with
 $\theta_{i+\frac12,j,k},\theta_{i,j+\frac12,k},\theta_{i,j,k+\frac12}$ computed by the proposed parametrized PCP flux limiter.
Then, the resulting scheme \eqref{eq:FD} is PCP under the CFL-type condition \eqref{eq:FDLFcfl}.
\end{theorem}

\subsubsection{Remarks}

The scheme \eqref{eq:FD} is only first-order accurate in time.
High-order SSP methods \cite{Gottlieb2009} can be used to repalce the forward Euler
time discretization in \eqref{eq:FD}, to achieve PCP scheme with high-order accuracy in time.
If the SSP Runge-Kutta (resp. multi-step) method is employed, the parametrized PCP flux limiter
should be used in each Runge-Kutta stage (resp. each time step).

The proposed PCP flux limiter does not depend on what numerical fluxes one uses, that is to say,
any high-order finite difference schemes for the GRHD equations in W-form \eqref{eq:GRHDconWKL}
can be modified into PCP schemes by the proposed parametrized PCP flux limiter.
%\end{remark}
%Besides,
%\begin{remark}

Although the parametrized PCP flux limiter is presented here for finite difference scheme, it is also applicable for high-order finite volume or DG methods to preserve the admissibility of approximate cell-averages. %Moreover, %in comparison to the simple scaling PCP limiter proposed in the last subsection,
%the PCP parametrized flux limiter can be very easily implemented for finite volume or DG schemes on nonuniform or unstructured mesh.
%\end{remark}

\section{Conclusions} \label{sec:conclude}

The paper designed high-order, physical-constraint-preserving (PCP) methods for
%established the general framework of designing high-order physical-constraint-preserving (PCP) methods, including PCP finite difference, finite volume and DG methods, for
the general relativistic hydrodynamic (GRHD) equations with a general equation of state.
It was built on the theoretical analysis of the admissible states of GRHD,
and two types of PCP limiting procedures enforcing the admissibility
of numerical solutions.
To overcome the difficulties arising from the strong nonlinearity contained in the physical constraints,
an ``explicit'' equivalent form of the admissible state set, ${\mathcal G}_\gamma$, was derived,
followed by several pivotal properties of ${\mathcal G}_\gamma$, including the convexity, scaling invariance and Lax-Friedrichs (LxF) splitting property.
It was discovered that the sets ${\mathcal G}_\gamma$ defined at different points in curved spacetime are inequivalent.
This invalidated the convexity of ${\mathcal G}_\gamma$ in analyzing PCP schemes.
To solve this problem, we used
a linear transformation to map the different ${\mathcal G}_\gamma$ into a common set ${\mathcal G}_*$,
which is also convex and exactly the admissible state set
in special RHD case. We then proposed
a new formulation (called W-form) of the GRHD equations to construct provably PCP schemes by taking advantages of the convexity of ${\mathcal G}_*$.
Under disretization on this W-form,
the first-order LxF scheme on general unstructured mesh was proved to be PCP, and
high-order PCP finite difference, finite volume and DG methods were designed via
two types of PCP limiting procedures.
%More validations and investigations of the proposed PCP methods via numerical experiments are ongoing work, and may be reported separately in further.
%The proposed methods and techniques can be useful for robust simulations in numerical relativity.
It is of particular significance to conduct more validations and investigations on the proposed PCP methods
via ultra-relativistic numerical experiments. This is our further work, which may be explored together with
computational astrophysicists.

\appendix

\section{\label{app:proofs}Proofs}

\subsection{\label{app:1} Proof of Lemma \ref{lem:equDef}}
\begin{proof}
The proof consists of two parts.

(1). Show that ${\bf  U}\in {\mathcal G} \Longrightarrow {\bf  U} \in {\mathcal G}_\gamma$.  When ${\bf  U} = (D,{\bf  m},E)^\top \in {\mathcal G}$ satisfy the constraints $\rho ({\bf  U}) > 0$, $p({\bf  U}) > 0$, $e({\bf  U})>0$ and $0\le v({\bf  U}) < 1$, then
\begin{align*}
&D = \frac{\rho }{{\sqrt {1 -   v^2 } }} > 0,
\\
&
E = \frac{{\rho h}}{{1 -  v^2 }} - p > \rho h - p = \rho (1 + e) > 0.
\end{align*}
Using \eqref{eq:hcondition1} gives
\begin{align*}
& E^2  - \left( {D^2  +  m_jm^j } \right)
\\
  & = \frac{1}{{1 - v^2 }}\left( {\left( {\rho h - p} \right)^2  - \rho ^2  - p^2 v^2 } \right)\\
    & >
    \frac{1}{{1 - v^2 }}\left( {\left( {\rho h - p} \right)^2  - \rho ^2  - p^2 } \right)
     \ge 0,
\end{align*}
which, along with $E>0$, further yield that
$$q_\gamma({\bf  U}) = E- \sqrt{D^2  +  m_j m^j } >0.$$
Therefore ${\bf  U}\in {\mathcal G}_\gamma$.

(2). Show that ${\bf  U}\in {\mathcal G}_\gamma \Longrightarrow {\bf  U}\in {\mathcal G}$. Consider the function
\begin{align*}
& \Psi^{[{\bf U}]} (p)  =   Dh\left( {p,  \rho^{[{\bf U}]}(p)  } \right)\sqrt {1 - \frac{{ {\bf  m} \bm \Upsilon {\bf  m}^\top }}{{(E + p)^2 }}}
\\
& \quad  +(E + p) \left( \frac{{ {\bf  m} \bm \Upsilon {\bf  m}^\top }}{{(E + p)^2}} -1 \right),~~p \in [0,+\infty),
\end{align*}
which is related to \eqref{eq:solvePgEOS}.
%where $\rho^{[\bm U]}(p)= D\sqrt {1 - |\bm  m|^2 /{(E + p)^2 }}$.
For given ${\bf  U} \in {\mathcal G}_\gamma$, we have
$\Psi^{[{\bf U}]} (p) \in C^1(\mathbb{R}^+)$ from $\rho^{[{\bf U}]}(p)  \in C^1(\mathbb{R}^+)$ and $e(\rho,p)\in C^1({\mathbb{R}}^+\times{\mathbb{R}}^+)$.
On the other hand, Eqs. \eqref{eq:EOS:h} and \eqref{eq:epto0} imply
$$
\mathop{\lim }\limits_{p \to 0^+ }  h\left( {p,  \rho^{[{\bf U}]}(p)  } \right) = 1, \quad
\mathop {\lim }\limits_{p \to +\infty }  e\left( {p,  \rho^{[{\bf U}]}(p)  } \right) = +\infty,
$$
which further gives
\begin{align*}
\begin{split}
& \mathop{\lim }\limits_{p \to 0^+ } \Psi^{[{\bf U}]} (p) = D \sqrt {1 - \frac{{ {\bf  m} \bm \Upsilon {\bf  m}^\top }}{{E ^2 }}}
+ \frac{{ {\bf  m} \bm \Upsilon {\bf  m}^\top }}{{E }} -E
\\
&\qquad ~~
= \left( D-\sqrt{E^2- {\bf  m} \bm \Upsilon {\bf  m}^\top }\right) \sqrt {1 - \frac{{ {\bf  m} \bm \Upsilon {\bf  m}^\top }}{{E ^2 }}} < 0,
\end{split}
\\
\begin{split}
& \mathop{\lim }\limits_{p \to +\infty } \Psi^{[{\bf U}]} (p) =
\mathop{\lim }\limits_{p \to +\infty }  D \left [1+ e \left( {p,  \rho^{[{\bf U}]}(p)  } \right) \right]
 \\
& \qquad \qquad
\times \sqrt {1 - \frac{{ {\bf  m} \bm \Upsilon {\bf  m}^\top }}{{(E + p)^2 }}}
+ \frac{{ {\bf  m} \bm \Upsilon {\bf  m}^\top }}{{E + p}} -E  = + \infty.
\end{split}
\end{align*}
According to the {\em intermediate value theorem}, $\Psi^{[{\bf U}]} (p)$ has at least one positive zero. In other words, there exist at least one positive solution to the algebraic equation $\Psi^{[{\bf U}]} (p)=0$ or \eqref{eq:solvePgEOS}.

We then indirectly show the uniqueness of positive zero of $\Psi^{[{\bf U}]} (p)$ via the proof by contradiction.
Assume that $\Psi^{[{\bf U}]} (p)$ has more than one positive zeros, and the smallest two are respectively denoted  by $p_1({\bf  U})$ and $p_2({\bf  U})$ with $p_2({\bf  U})>p_1({\bf  U})>0$. Then the equivalence between the equation $\Psi^{[{\bf U}]} (p) =0$ and \eqref{eq:solvePgEOS} leads to
\begin{equation}\label{eq:gEOSproof5}
D  h\left( {p_i, \rho^{[{\bf U}]} (p_i)} \right) =(E+p_i)\sqrt{ 1-\frac{ {\bf  m} \bm \Upsilon {\bf  m}^\top }{(E+p_i)^2} },
\end{equation}
for $i = 1, 2$. It follows from the constraints in \eqref{EQ-adm-set02} that $h\left( {p_i, \rho^{[{\bf U}]} (p_i)} \right) >0$.
Combining \eqref{eq:gEOSC}, we further get
\begin{equation}\label{eq:proof66}
\frac{{\partial h}}{{\partial p}} \left(p_i, \rho^{[{\bf U}]}(p_i) \right)>\frac{1}{\rho^{[{\bf U}]} (p_i)}>0.
\end{equation}
In the following, we utilize \eqref{eq:gEOSC}, \eqref{eq:gEOSproof5} and \eqref{eq:proof66}, to evaluate the lower-bound of the derivative $\frac{ {\rm d} \Psi^{[{\bf U}]} (p)}{{\rm d} p}$ at $p_1$ and $p_2$,
which can be expressed as
{\small
\begin{equation*}
\begin{split}
\frac{ {\rm d} \Psi^{[{\bf U}]}} {{\rm d} p } (p)&=
D \bigg\{  {\frac{{\partial h }}{{\partial p}}} \left(p, \rho^{[{\bf U}]}(p)  \right)  \sqrt {1 - \frac{{ {\bf  m} \bm \Upsilon {\bf  m}^\top }}{{(E + p)^2 }}} + \frac{{D {\bf  m} \bm \Upsilon {\bf  m}^\top }}{{(E + p)^3 }}
\\
& \quad  \times {\frac{{\partial h}}{{\partial \rho }}}   \left(p, \rho^{[{\bf U}]}(p)  \right)   \bigg\}
+
\frac{{D {\bf  m} \bm \Upsilon {\bf  m}^\top }}{{(E + p)^3 }} h\left( {p,  \rho^{[{\bf U}]}(p)  } \right)
\\
& \quad
\times
 \left( {1 - \frac{{ {\bf  m} \bm \Upsilon {\bf  m}^\top  }}{{(E + p)^2 }}} \right)^{ - \frac{1}{2}}
- \frac{{ {\bf  m} \bm \Upsilon {\bf  m}^\top }}{{(E + p)^2 }} -1.
\end{split}
\end{equation*}
}Specifically, we have
{\small
\begin{equation*}
\begin{aligned}
 & \frac{ {\rm d} \Psi^{[{\bf U}]}} {{\rm d} p }(p_i)
 \overset{\eqref{eq:gEOSC}}{>} \Big( D \frac{{\partial h}}{{\partial p}} \big( {p_i,  \rho^{[{\bf U}]}(p_i)  } \big) \Big)
  \Bigg(\sqrt {1 - \frac{{ {\bf  m} \bm \Upsilon {\bf  m}^\top  }}{{(E + p_i)^2 }}}
 \\
 &
  - \frac{{ D {\bf  m} \bm \Upsilon {\bf  m}^\top h \big( {p_i,  \rho^{[{\bf U}]}(p_i)  } \big)  }}{{(E + p_i)^3  }}     \Bigg)
 +
  \frac{{2 D {\bf  m} \bm \Upsilon {\bf  m}^\top }}{{(E + p_i)^3 }}h\Big( {p_i,  \rho^{[{\bf U}]}(p_i)  } \Big)
   \\
   & \qquad \times
   \left( {1 - \frac{{ {\bf  m} \bm \Upsilon {\bf  m}^\top }}{{(E + p_i)^2 }}} \right)^{ - \frac{1}{2}}
- \frac{{ {\bf  m} \bm \Upsilon {\bf  m}^\top }}{{(E + p_i)^2 }} -1
\\
%\frac{ {\rm d} \Psi^{[\bm U]} } { {\rm d} p } (p_i)
&  \overset{\eqref{eq:gEOSproof5}}{=} D  \left( {1 - \frac{{ {\bf  m} \bm \Upsilon {\bf  m}^\top }}{{(E + p_i)^2 }}} \right)^{ \frac{3}{2}}
   \frac{{\partial h}}{{\partial p}} \big(p_i, \rho^{[{\bf U}]}(p_i) \big)
 +  \frac{{ \bm  m \bm   \Upsilon  \bm  m^\top }}{{(E + p_i)^2 }} -1
   \\
   & \overset{\eqref{eq:proof66}}{>}  D  \left( {1 - \frac{{ {\bf  m} \bm \Upsilon {\bf  m}^\top }}{{(E + p_i)^2 }}} \right)^{ \frac{3}{2}}
    \frac{1}{\rho^{[{\bf U}]} (p_i)} +  \frac{{ {\bf  m} \bm \Upsilon {\bf  m}^\top }}{{(E + p_i)^2 }} -1
  \\ & ~=0,
    \quad i = 1, 2.
\end{aligned}
\end{equation*}
}This indicates
\begin{equation*}
\begin{split}
\mathop {\lim }\limits_{\delta_p \to 0} \frac{{\Psi^{[{\bf U}]} (p_i  + \delta_p) }}{\delta_p} &= \mathop {\lim }\limits_{\delta_p \to 0} \frac{{\Psi^{[{\bf U}]} (p_i  + \delta_p) - \Psi^{[{\bf U}]} (p_i )}}{\delta_p}
\\
&=  \frac{ {\rm d} \Psi^{[{\bf U}]} } { {\rm d} p }  (p_i ) > 0,\quad i = 1, 2,
\end{split}
\end{equation*}
where $\Psi^{[{\bf U}]} (p_i )=0$ is used in the first equality. According to the $(\varepsilon, \delta)$-definition of limit, for $\varepsilon_i= \frac{1}{2} \frac{ {\rm d} \Psi^{[{\bf U}]} } { {\rm d} p } (p_i ) >0 $, there exists $\delta_i >0$ such that
$$\left | \frac{{ \Psi^{[{\bf U}]} (p_i  + \delta_p)}}{\delta_p} - \frac{ {\rm d} \Psi^{[{\bf U}]} } { {\rm d} p } (p_i ) \right| < \varepsilon_i,\quad \forall~\delta_p \in ( -\delta_i, \delta_i ).$$
It follows that
$$  \varepsilon_i < \frac{{ \Psi^{[{\bf U}]} (p_i  + \delta_p) }}{\delta_p} < 3 \varepsilon_i ,\quad \forall \delta_p \in ( -\delta_0, \delta_0 ),$$
where $\delta_0 = \min \left\{  \delta_1,\delta_2, \frac{p_2-p_1}{2} \right\}>0$.
We therefore have
$$\big(p_1 +  {\delta_0}/{2},~p_2 -  {\delta_0}/{2}\big) \subset (p_1,p_2), $$
and
\begin{align*}%\label{eq:gEOSproof7}
\Psi^{[{\bf U}]} ( p_1 +  {\delta_0}/{2} ) > 0,\quad  \Psi^{[{\bf U}]} ( p_2 -  {\delta_0}/{2} ) <0.
\end{align*}
 It implies that the function $\Psi^{[{\bf U}]} (p)$ has zero in the interval $\big( p_1 +  {\delta_0}/{2},p_2 -  {\delta_0}/{2}\big)$, according to the {\em intermediate value theorem}. This contradicts our assumption that $p_1$ and $p_2$ are the smallest two positive zeros of $\Psi^{[\bm U]} (p)$.
 Hence the assumption does not hold. In other words, $\Psi^{[{\bf U}]}  (p)$ has unique positive zero, denoted by $p({\bf  U})>0$.
 Substituting the positive pressure $p({\bf  U})$ into \eqref{eq:solveVRHOgEOS} and using the constraints in \eqref{EQ-adm-set02} give
\begin{align*}
&
{v( {\bf  U} )}  = \frac{{ \sqrt{ m_j m^j  } }}{{E + p({\bf  U})}} <  \frac{{  \sqrt{ {\bf  m} \bm \Upsilon {\bf  m}^\top } }}{E} < 1,
\\
&
\rho ( {\bf  U} ) = D\sqrt {1 -  { v^2( {\bf  U} )}   }  > 0.
\end{align*}
For any $p,\rho \in \mathbb{R}^+$, the condition \eqref{eq:gEOSC} implies $\pt_p e(p,\rho)>0$, and further yields
$$
e({\bf  U}) = e( p({\bf  U}),\rho({\bf  U}) ) >  \mathop{\lim }\limits_{p \to 0^+ }  e( p,\rho({\bf  U}) ) =   0,
$$
where \eqref{eq:epto0} is used in the last equality. In conclusion, ${\bf  U} \in {\mathcal G}$. The proof is completed.
\end{proof}

%\begin{remark}
%For the ideal EOS \eqref{eq:iEOS}, the part (2) in the above proof can be simplified. It is similar to the special RHD case in \cite{WuTang2015}.
%\end{remark}

\subsection{\label{app:2}Proof of Lemma \ref{lam:convexGgamma}}

\begin{proof}
%Obviously, ${\mathcal G}_\gamma$ is open.
Denote $\lambda {\bf U}' + (1-\lambda) {\bf U}''$ by
${\bf  U} _\lambda = ( D_\lambda, {\bf  m}_\lambda, E_\lambda )^\top $, then
$$
D_\lambda = \lambda D' + (1-\lambda) D'' > 0,
$$
and
\begin{equation*}
\begin{split}
q_\gamma (  {\bf  U}_\lambda )
&= q ( \bm  \Sigma {\bf  U}_\lambda ) = q \big(   \lambda \bm  \Sigma {\bf  U}' +(1-\lambda) \bm  \Sigma {\bf  U}'' \big) \\
&  \ge \lambda q \big( \bm  \Sigma {\bf  U}'  \big) +(1-\lambda) q \big( \bf  \Sigma {\bf  U}''  \big) \\
 & = \lambda q_\gamma \big( {\bf  U}' \big) +(1-\lambda) q_\gamma \big(  {\bf  U}'' \big) > 0,
\end{split}
\end{equation*}
where $\bm  \Sigma= {\mathrm {diag}}\{1,\bm \Upsilon^{\frac12},1 \}$, and the concavity \cite{WuTang2015} of the function
%\begin{equation*}%\label{eq:Defq}
$
q({\bf  U}) = E - \sqrt{D^2 + |{\bf  m}|^2}
$
%\end{equation*}
is used. This shows $ {\bf  U}_\lambda \in {\mathcal G}_\gamma$ and the convexity of ${\mathcal G}_\gamma$.
With the fact that ${\mathcal G}_\gamma$ is open, we complete the proof.
\end{proof}

\subsection{\label{app:3}Proof of Lemma \ref{lam:propertyG}}

\begin{proof}
The scaling invariance can be directly verified by the definition of ${\mathcal G}_\gamma$.
In the following, we prove the LxF splitting property via two steps.

(1). Show that ${\bf  U} \pm \varrho_\xi^{-1}{{ \xi_j {\bf  F}^j ({\bf  U})}} \in \overline {\mathcal G}_\gamma$.
We would like to split it as a form of convex combination
\begin{equation}\label{eq:splitGRHD}
{\bf  U} \pm \varrho_\xi ^{-1}\xi_j  {\bf  F}^j ({\bf  U}) =
\frac12 \bigg( \frac{ 2 \hat \varrho_\xi }{ \varrho_\xi } {\bf  U}^\pm \bigg) + \frac12  \widetilde {\bf  U}^\pm ,
\end{equation}
with
\begin{equation*}
\begin{split}
&{\bf  U}^\pm = {\bf  U} \pm {\hat \varrho_\xi }^{-1}\xi_j {{ \hat{ \bf F}^j ({\bf  U})}} , \\
&\widetilde {\bf  U}^\pm = 2 \frac{ | \xi_j \beta^j | \mp  \big( \xi_j \beta^j \big)  }{\alpha \varrho_\xi }     {\bf U},
\end{split}
\end{equation*}
where $\widetilde {\bf  U}^\pm \in \overline{\mathcal G}_\gamma$ due to the scaling invariance,
\begin{align*}
 \xi_j {{ \hat {\bf F}^j ({\bf  U})}}  = \left( D \xi_j v^j, (\xi_j v^j) {\bf  m} + p \bm  \xi, (E+p) \xi_j v^j   \right)^\top,
\end{align*}
and the positive quantity $\hat \varrho_\xi = \varrho_\xi - |\xi_j \beta^j|/\alpha $ equals $\sqrt{\xi_j\xi^j}$ for general EOS, while for ideal EOS with sharper $\varrho_\xi$,
\begin{align*}
& \hat \varrho_\xi   = \frac{1}{ {{1 - v^2 c_s^2 }}} \Big\{ |\xi_j v^j | (1 - c_s^{ 2} ) + c_s W^{ - 1}
 \\
&\qquad \times \sqrt {  (1-v^2c_s^2)(\xi_j\xi^j)-(1-c_s^2) (\xi_j v^j)^2  } \Big\}.
\end{align*}
With the help of Lemma \ref{lam:convexGgamma} and the scaling invariance of ${\mathcal G}_\gamma$, the form in \eqref{eq:splitGRHD}
indicates that it suffices to show
\begin{equation}\label{eq:UpminG}
{\bf  U}^\pm \in {\mathcal G}_\gamma.
\end{equation}
To this end, we denote ${\bf  U}^\pm=:(D^\pm,{\bf  m}^\pm,E^\pm)^\top $, then
\begin{equation}\label{eq:DE}
\begin{split}
&D^ \pm  = D\left( {1 \pm \frac{{\xi_j v^j }}{ \hat \varrho_\xi }} \right),\\
& E^ \pm   = \rho hW^2\left( 1 \pm  \frac{\xi_j v^j }{\hat \varrho_\xi }  \right) - p,
\end{split}
\end{equation}
and
\begin{equation} \label{EQ-wkl03}
{
\small
\begin{aligned}
& \left( {D^ \pm  } \right)^2  +  \gamma^{ij} m^\pm_i m^\pm_j    - \left( {E^ \pm  } \right)^2
  = \left( {1 \pm  \frac{{\xi_j v^j }}{\hat \varrho_\xi }  } \right)^2  W^2
\\
&
\qquad \times  \Big( {\rho ^2  + p^2  - \left( \rho h - p\right)^2 } \Big)
+ p^2 \bigg( {\frac{ \xi_j \xi^j }{{  {\hat \varrho_\xi} ^2  }} - 1} \bigg).
\end{aligned}
}
\end{equation}
With the formulas \eqref{eq:DE}--\eqref{EQ-wkl03}, we shall prove \eqref{eq:UpminG} in the following for two cases separately, i.e.,
the general EOS case, and the ideal EOS case with sharper $\varrho_\xi$. We will always employ the Cauchy-Schwarz type inequality
\begin{equation}\label{eq:CSieq}
\begin{aligned}
(\xi_j v^j)^2 &= \big(\bm  \xi \bm  \Upsilon \bm  v^\top \big)^2 = \big( (\bm  \xi \bm  \Upsilon ^{\frac12} ) (   \bm  v \bm  \Upsilon^{\frac12})^\top  \big)^2
\\
& \le \big( (\bm  \xi \bm  \Upsilon^{\frac12} ) (   \bm  \xi \bm  \Upsilon^{\frac12})^\top  \big) \big( (\bm  v \bm  \Upsilon^{\frac12} ) (   \bm  v \bm  \Upsilon^{\frac12})^\top  \big) \\
&= \big(\bm  \xi \bm  \Upsilon \bm  \xi^\top \big) \big(\bm  v \bm \Upsilon \bm  v^\top \big)
= v^2 (\xi_j \xi^j).
\end{aligned}
\end{equation}
First consider general EOS. Using \eqref{eq:DE}--\eqref{eq:CSieq} and \eqref{eq:hcondition1} gives
\begin{align*}
&D^ \pm  \ge D \left( {1 - \frac{{|\xi_j v^j| }}{ \sqrt{\xi_j \xi^j} }} \right) \ge D (1-v) > 0,
\\[1mm]
& E^ \pm  \ge \rho hW^2\left( 1 -  \frac{|\xi_j v^j| }{ \sqrt{\xi_j \xi^j} }  \right) - p
 \ge  \frac{\rho h} {1+v} - p
\\
& \quad ~ > \frac{\rho h} { 2} -p  \overset {\eqref{eq:hcondition1}} {\ge} \frac{1}{2}\big( \sqrt{\rho^2+p^2} - p \big) >0,\\
\end{align*}
and
\begin{align*}
& \left( {D^ \pm  } \right)^2  +  \gamma^{ij} m^\pm_i m^\pm_j    - \left( {E^ \pm  } \right)^2
\\
&
= \bigg( {1 \pm \frac{{\xi_j v^j }}{ \sqrt{\xi_j \xi^j} } } \bigg)^2 W^2 \left[ {\rho ^2  + p^2  - \left( \rho h - p\right)^2 } \right]
\overset {\eqref{eq:hcondition1}} {\le} 0,
\end{align*}
which immediately imply \eqref{eq:UpminG}. Then we focus on the ideal EOS case with sharper $\varrho_\xi$. From \eqref{eq:CSieq} and $0<c_s<1$, we derive
\begin{align*}
&  \sqrt {  (1-v^2c_s^2)(\xi_j\xi^j)-(1-c_s^2) (\xi_j v^j)^2  }
\\
& \quad =
  \sqrt { \xi_j\xi^j -  (\xi_j v^j)^2  -  c_s^2 \big( v^2 (\xi_j\xi^j) - (\xi_j v^j)^2  \big)  }
\\
&\quad \ge  \sqrt { \xi_j\xi^j -  (\xi_j v^j)^2  -  \big( v^2 (\xi_j\xi^j) - (\xi_j v^j)^2  \big)  }
\\
& \quad =  W^{ - 1}\sqrt{\xi_j\xi^j},
\end{align*}
and further
 $$
\hat \varrho_\xi \geq \frac
    {{|\xi_j v^j|} (1 - c_s^{ 2} ) + c_s (1 - v^2 ) \sqrt{\xi_j \xi^j}}{  1 - v^2 c_s^2  }.
 $$
This implies
 \begin{align*}
 \begin{split}
 1 - \frac{{\left| {\xi_j v^j } \right|}}{ \hat \varrho _\xi }
 &\ge 1 - \frac{ \left| \xi_j v^j  \right| (1 - v^2 c_s^2 ) }
   {{|\xi_j v^j|} (1 - c_s^{ 2} ) + c_s (1 - v^2 ) \sqrt{\xi_j \xi^j}  }
 \\
 & = \frac{{W^{ - 2} c_s \big( {\sqrt{\xi_j \xi^j} - \left| {\xi_j v^j } \right|c_s } \big)}}
   {{|\xi_j v^j|} (1 - c_s^{ 2} ) + c_s (1 - v^2 ) \sqrt{\xi_j \xi^j}  }\\[2mm]
 & \overset{\eqref{eq:CSieq}} {\ge }
 \frac{{W^{ - 2} c_s \big( {\sqrt{\xi_j \xi^j} - \left| {\xi_j v^j } \right|c_s } \big)}}
   {{|\xi_j v^j|} (1 - c_s^{ 2} ) + c_s \Big(1 - \frac{|\xi_j v^j|^2}{ \xi_j \xi^j } \Big) \sqrt{\xi_j \xi^j}  }
\\
&= \frac{ W^{-2} c_s }{ c_s +  \frac{|\xi_j v^j|}{ \sqrt{\xi_j \xi^j} } }
   \overset{\eqref{eq:CSieq}} {\ge }
  \frac{{W^{ - 2} c_s }}{{ c_s + v }} ,
  \end{split}
\end{align*}
and further gives
\begin{align}
 1 \pm \frac{{ {\xi_j v^j } }}{ \hat \varrho _\xi } \ge  1 - \frac{{\left| {\xi_j v^j } \right|}}{ \hat \varrho _\xi }
\ge \frac{{W^{ - 2} c_s }}{{ c_s + v }}
> \frac{{W^{ - 2} c_s }}{{ c_s + 1 }} >0.\label{EQ-wkl01}
\end{align}
It follows, along with \eqref{eq:DE}, that $D^ \pm>0$ and
\begin{equation*}
\begin{split}
 E^ \pm  &
  \overset{\eqref{EQ-wkl01}}{>} \frac{{\rho hc_s }}{{1 + c_s }} - p
  = p\left( {\frac{\Gamma }{{c_s (1 + c_s )}} - 1} \right)
 \\ & > p\left( {\frac{\Gamma }{{\Gamma  - 1 + \sqrt {\Gamma  - 1} }} - 1} \right)
\ge
   0,
\end{split}
\end{equation*}
where $0<c_s \le \sqrt{\Gamma-1}$ and $\Gamma \in (1,2]$ are used. Note that $\hat \varrho_\xi$ is a positive solution to the following quadratic equation
\begin{align*}
& (1 - v^2 c_s^2 ){\hat \varrho_\xi}^2  - 2\left| {\xi_j v^j } \right|(1 - c_s^{ 2} ) \hat \varrho_\xi
\\
& \qquad + { (1-c_s^2)(\xi_j v^j)^2 - c_s^2 (1-v^2) ( \xi_j \xi^j )}  = 0,
\end{align*}
which is equivalent to
\begin{align}\label{EQ-wkl04}
\left( { \xi_j \xi^j - { \hat \varrho_\xi}^2 } \right)c_s^2  = W^2 \left( {  \hat \varrho_\xi  - \left| { \xi_j v^j } \right|} \right)^2 (1 - c_s^2 ).
\end{align}
It implies $ \hat \varrho_\xi < \sqrt{\xi_j \xi^j}$. Using \eqref{EQ-wkl01}--\eqref{EQ-wkl04} for \eqref{EQ-wkl03} gives
{\small
\begin{equation*}
 \begin{aligned}
&
 \left( {D^ \pm  } \right)^2  + \gamma^{ij} m^\pm_i m^\pm_j   - \left( {E^ \pm  } \right)^2
 \overset{\eqref{EQ-wkl01} } \le
 \left( {1 -  \frac{{|\xi_j v^j| }}{\hat \varrho_\xi }  } \right)^2 W^2
  \\
& \qquad \times \left( {\rho ^2
 + p^2  - \Big( {\rho  + \frac{p}{{\Gamma  - 1}}} \Big)^2 } \right)
+ p^2 \bigg( {\frac{\xi_j \xi^j}{{ {\hat \varrho_\xi} ^2 }} - 1} \bigg) \\
& \overset{\eqref{EQ-wkl04}}{=}
 \bigg( {\frac{\xi_j \xi^j}{{ \hat \varrho_\xi }^2 } - 1} \bigg)
  \bigg(
 \frac{{c_s^2 }}{{1 - c_s^2 }}
 \Big( {p^2  - \frac{{2\rho p}}{{\Gamma  - 1}} - \frac{{p^2 }}{{( \Gamma  - 1)^2 }}} \Big)
 + p^2 \bigg)
 \\
&  = \bigg( { \frac{\xi_j \xi^j}{{ \hat \varrho_\xi }^2 }  - 1} \bigg)\frac{{p^2 }}{{(1 - c_s^2 )\left( {\Gamma  - 1} \right)}}\bigg( {\Gamma  - 1 - c_s^2 \Big( {\frac{1}{{\Gamma  - 1}} + \frac{{2\rho }}{p}} \Big)} \bigg) \\
&  \le \bigg( { \frac{\xi_j \xi^j}{{ \hat \varrho_\xi }^2 }  - 1} \bigg)\frac{{p^2 }}{{(1 - c_s^2 )\left( {\Gamma  - 1} \right)}}\bigg( {1 - c_s^2 \Big( {\frac{1}{{\Gamma  - 1}} + \frac{{2\rho }}{p}} \Big)} \bigg) \\
&  = \bigg( { \frac{\xi_j \xi^j}{{ \hat \varrho_\xi }^2 }  - 1} \bigg)
\frac{{p^2 }}{{(1 - c_s^2 )\left( {\Gamma  - 1} \right)}}\cdot \frac{1 - 2\Gamma}{h}
< 0,
\end{aligned}
\end{equation*}
}where $\Gamma \ge 2$ and $\Gamma <1$ are respectively used in the last two inequalities.
In conclusion, \eqref{eq:UpminG} holds for ideal EOS case with sharper $\varrho_\xi$.

The proof of part (1) is completed.

(2). For any $\eta>\varrho_\xi$, we have
{\small
\begin{equation*}
 {\bf  U} \pm {\eta}^{-1}\xi_j {{ {\bf  F}^j ( {\bf  U})}} =  \left( 1 - \frac{ \varrho_\xi }{\eta} \right) {\bf  U} +
\frac{ \varrho_\xi }{\eta} \Big( {\bf  U} \pm \varrho_\xi ^{-1}\xi_j { { {\bf  F}^j ( {\bf  U} )}} \Big).
\end{equation*}
}It follows from Lemma \ref{lam:convexGgamma} and the deduction proved in part (1) that $ {\bf  U} \pm {\eta}^{-1}\xi_j {{ {\bf  F}^j ( {\bf  U})}} \in {\mathcal G}_\gamma $.

The proof is completed.
\end{proof}

In the end we give a remark on Lemma \ref{lam:propertyG}.
For general EOS, one can also choose %(according to \eqref{EQ-wkl03})
  $$
\varrho_\xi   =  \frac{\varsigma |\xi_j v^j| + \sqrt{ (\varsigma+1) \xi_j \xi^j - \varsigma (\xi_j v^j)^2 } } { \varsigma + 1 }  + \frac{ | \xi_j \beta^j | }{\alpha},
  $$
to establish the LxF splitting property in Lemma \ref{lam:propertyG},
where $\varsigma = \big((\rho h-p)^2-\rho^2-p^2\big)W^2/p^2 \ge 0$. This choice of $\varrho_\xi$
is smaller/sharper than that in \eqref{eq:WKLvarrhoGEOS}, but generally not an upper bound
  of the spectral radius of $\pt ( \xi_j {{ {\bf  F}^j ( {\bf  U} )}} )/\pt {\bf  U}$.

\subsection{\label{app:4}Proof of Theorem \ref{thm:firstLF}}

Before proving Theorem \ref{thm:firstLF}, we first introduce a lemma.

\begin{lemma}\label{lam:LF1}
If $\overline{{\bf  W}}_k^n \in {\mathcal G}_*$ for all $k$, then for any $\delta_t$ satisfying
\begin{equation}\label{eq:cflLF}
0<\max \limits_{k} \frac{ \delta_t }{2|{\mathcal I}_k|}  \sum_{ j \in {\mathcal N}_k }
a_{kj}
 \big| {\mathcal E}_{kj} \big| < 1,
\end{equation}
it holds
$$
\widetilde { {\bf  W} }_k^{n}=  \overline{ {\bf  W} }_k^{n} - \frac{ \delta_t }{|{\mathcal I}_k|}  \sum_{ j \in {\mathcal N}_k }
 \big| {\mathcal E}_{kj} \big| \widehat{\bf  H}_{kj} \in {\mathcal G}_*.
$$
\end{lemma}
\begin{proof}
Using the identity
$$
\sum_{ j \in {\mathcal N}_k }
 \big| {\mathcal E}_{kj} \big| \left(  \bm  {\xi}_{kj} \cdot {\bf Z} \right)
=  \int_{{\mathcal I}_k}   \frac{\partial  Z^\ell }{\partial x^\ell}         {\rm d} \bm  x  \equiv 0,
$$
for any constant vector ${\bf Z} = (Z^1,Z^2,Z^3)$, we reformulate $\widetilde { {\bf  W} }_k^{n} $ as
\begin{equation}\label{eq:LFrew}
\begin{split}
 \widetilde { {\bf  W} }_k^{n}  &=
 \Big( 1 - \frac{\delta_t }{2|{\mathcal I}_k|}  \sum_{ j \in {\mathcal N}_k }
 a_{kj}
 | {\mathcal E}_{kj} | \Big) \overline {\bf  W}_k^{n}
 \\
& \quad
 + \frac{\delta_t }{2|{\mathcal I}_k|}  \sum_{  j \in {\mathcal N}_k } a_{kj}
 | {\mathcal E}_{kj} | {\bm  \Pi}_{kj},
\end{split}
\end{equation}
with
%\begin{align*}
%2 \Pi_{k_j}&= \overline{\bm  W}_k^m  + \eta_{k_j}^{-1}( \sqrt{-g}\bm  \Sigma)_k^m \sum_{\ell=1}^3 n^{(\ell)}_{k_j} \bm  F^{\ell} (\overline {\bm  w}_k^m )
%+ \overline{\bm  W}_{k_j}^m + \eta_{k_j}^{-1} ( \sqrt{-g} \bm  \Sigma)_{k_j}^m \sum_{\ell=1}^3 n^{(\ell)}_{k_j} \bm  F^{\ell} (\overline {\bm  w}_{k_j}^m )
%\\
%& = ( \sqrt{\gamma} )_k^m \bm  \Sigma_k^m \bigg(
%{\bm  U} (\overline {\bm  w}_k^m)  + \eta_{k_j}^{-1} \alpha_k^m \sum_{\ell=1}^3 n^{(\ell)}_{k_j} \bm  F^{\ell} (\overline {\bm  w}_k^m ) \bigg)
%\\
%&\quad + ( \sqrt{\gamma} )_{k_j}^m \bm  \Sigma_{k_j}^m \bigg(  \bm  U( \overline{\bm  w}_{k_j}^m ) + \eta_{k_j}^{-1} \alpha_{k_j}^m \sum_{\ell=1}^3 n^{(\ell)}_{k_j} \bm  F^{\ell} (\overline {\bm  w}_{k_j}^m ) \bigg).
%\end{align*}
\begin{align*}
{\bm  \Pi}_{kj} = \overline{ {\bf  W} }_j^n - a_{kj}^{-1}  {\xi}_{kj,\ell}   {\bf  H}^{\ell} \big(\overline {\bf  W}_j^n \big).
\end{align*}
Thanks to Lemma \ref{lam:LxF} and the condition \eqref{eq:defLFa}, one has $\bm  \Pi_{kj} \in \overline{\mathcal G}_*$.
Thus the form \eqref{eq:LFrew} is a convex combination under the condition \eqref{eq:cflLF}.
The proof is completed by Lemma \ref{lam:convex}.
\end{proof}

Based on this lemma, the proof of Theorem \ref{thm:firstLF} is given as follows.

\begin{proof}
Here the induction argument is used for time level number $n$. Assume $\overline{\bf  W}_k^n \in {\mathcal G}_*$ for all $k$,
we then prove that $\overline{\bf W}_k^{n+1} $ computed by \eqref{eq:1orderLF} also belongs to ${\mathcal G}_*$.
The scheme \eqref{eq:1orderLF} can be rewritten as
\begin{align*}
& \overline{\bf  W}_k^{n+1} = \vartheta \bigg(  \overline{\bf W}_k^{n} - \frac{ \delta_t }{ |{\mathcal I}_k|}  \sum_{ j \in {\mathcal N}_k }
 \big| {\mathcal E}_{kj} \big| \widehat{\bf  H}_{kj}
 \bigg)
 \\
 &\quad
 + \Big( (1-\vartheta)  \overline{\bf  W}_k^{n} + \Delta t_n {\bf  S} \big( \overline {\bf  W}_k^n \big)  \Big) =: \vartheta \Xi_H + \Xi_S,
\end{align*}
where $\delta_t = \Delta t_n/\vartheta$, and
\begin{align*}
\vartheta & = \dfr{ \frac{1 }{2|{\mathcal C}_k|}  \sum_{ j \in {\mathcal N}_k }
a_{kj}
 \big| {\mathcal E}_{kj} \big| }
{
\frac{1 }{2|{\mathcal C}_k|}  \sum_{ j \in {\mathcal N}_k }
a_{kj}
 \big| {\mathcal E}_{kj} \big|
 + \lambda_{\rm S}(\overline{\bf W}_k^n) } \in (0,1].
\end{align*}
Under the condition \eqref{eq:cflLFfs}, we know that $\delta_t$ satisfies \eqref{eq:cflLF}
and thus have $\Xi_H \in {\mathcal G}_*$ by Lemma \ref{lam:LF1}.
We then show $\Xi_S \in \overline{\mathcal G}_*$ as follows.
\begin{itemize}[\hspace{0em}$\bullet$]
  \item If $q\big(  {\bf S} ( \overline {\bf W}_k^n )  \big) \ge 0$, then $\lambda_{\rm S}=0$ and $\vartheta=1$,
which yields $\Xi_S = \Delta t_n  {\bf S} \big( \overline {\bf W}_k^n \big) $. The first component of $\Xi_S$ is zero, and $q(\Xi_S)=\Delta t_n q\big( {\bf S} ( \overline {\bf W}_k^n ) \big) \ge 0$. Hence $\Xi_S \in \overline{\mathcal G}_*$.
  \item If $q\big( {\bf S} ( \overline {\bf W}_k^n ) \big) < 0$, then $ {\bf S} \big( \overline {\bf W}_k^n \big)  \notin \overline{\mathcal G}_*$.
  Thanks to the convexity of ${\mathcal G}_*$, Eq. \eqref{eq:LFs} has unique positive solution $\lambda_{\rm S}$.
      This implies $\vartheta \in (0,1)$, and
      $$ \overline{\bf W}_k^n + \lambda {\bf S} \big( \overline {\bf W}_k^n \big) \in {\mathcal G}_*,~~\mbox{for any}~\lambda \in [0,\lambda_{\rm S}^{-1}).$$
      Under the condition \eqref{eq:cflLFfs}, $\Delta t_n /(1-\vartheta) < \lambda_{\rm S}^{-1} $.
      It follows from the scaling invariance of ${\mathcal G}_* $ that
      $$ \Xi_S =  (1-\vartheta) \Big( \overline{\bf W}_k^n + \frac{\Delta t_n}{1-\vartheta} {\bf S} \big( \overline {\bf W}_k^n \big) \Big) \in {\mathcal G}_*
      \subset \overline{\mathcal G}_* .$$
\end{itemize}
Thanks to the scaling invariance, the above deductions imply $2 \vartheta \Xi_H \in {\mathcal G}_*$ and $2\Xi_S\in \overline {\mathcal G}_*$.
With Lemma \ref{lam:convex}, we then have
 $\overline {\bf W}_{k}^{n+1} = \frac{1}{2} \cdot 2 \vartheta \Xi_H  + \frac{1}{2} \cdot 2\Xi_S \in {\mathcal G}_*$.
\end{proof}

\subsection{\label{app:5}Proof of Theorem \ref{thm:FVDGPCP}}

\begin{proof}
Here the induction argument is used for time level number $n$. Assume $\overline{\bf  W}_{ijk}^n \in {\mathcal G}_*$ for all $i,j,k$,
we then show that $\overline{\bf W}_{i,j,k}^{n+1} $ computed by \eqref{eq:cellaverage} also belongs to ${\mathcal G}_*$.
Define
\begin{align*}
& {\bf{L}}_{ijk}^{H} ({\bf W}^n( \bm x )) := {\bf{L}}_{ijk} ({\bf W}^n( \bm x )) - \overline{\bf S}_{ijk}^n,\\
& \vartheta := \big( \Delta_\ell^{-1} a_\star^{(\ell)} \big) / \big( \Delta_\ell^{-1} a_\star^{(\ell)}  + \hat \omega_1 \lambda_{\rm S} \big) \in (0,1],
\end{align*}
then
$$\overline {\bf W}_{ijk}^{n+1} = \vartheta \Xi_H +  \Xi_S,$$
with
\begin{align*}
&\Xi_H = \overline {\bf W}_{ijk}^{n} + \vartheta^{-1} \Delta t_n {\bf{L}}_{ijk}^{H} \big({\bf W}^n( \bm x )),
\\
&\Xi_S = (1- \vartheta) \overline {\bf W}_{ijk}^{n} + \Delta t_n \overline{\bf S}_{ijk}^n.
\end{align*}
The proof of $\overline{\bf W}_{i,j,k}^{n+1} \in {\mathcal G}_* $ is divided into three parts.

(1). First prove $\Xi_H \in {\mathcal G}_*$. This part will always employ Einstein's summation convention for indices $\mu$ and $\nu$ running from
0 to ${\tt Q}$.
The exactness of the $\tt L$-point Gauss-Lobatto quadrature rule
and the $\tt Q$-point Gauss quadrature rule yields
\begin{align} \nonumber
\overline{\bf W}_{ijk}^n &= \frac{1}{\Delta_1 \Delta_2 \Delta_3} \int_{{\mathcal I}_{ijk}} {\bf W}_{ijk}^n ({\tt x},{\tt y},{\tt z})~{\rm d}{\tt x}{\rm d}{\tt y}{\rm d}{\tt z}
\\ \nonumber
&
=  \sum \limits_{\delta = 1}^{\tt L}  \hat \omega_\delta  \omega_\mu \omega_\nu     {\bf W}_{ijk}^n\big(\hat {\tt x}_i^{(\delta)},{\tt y}_j^{(\mu)},{\tt z}_k^{(\nu)}\big)
\\ \nonumber
&
=  \sum \limits_{\delta = 2}^{{\tt L}-1} \hat \omega_\delta  \omega_\mu \omega_\nu     {\bf W}_{ijk}^n\big(\hat {\tt x}_i^{(\delta)},{\tt y}_j^{(\mu)},{\tt z}_k^{(\nu)}\big)
\\
&
\quad + \hat \omega_1    \omega_\mu \omega_\nu
\Big(
 {\bf W}^{+,\mu,\nu}_{i-\frac{1}{2},j,k} + {\bf W}^{-,\mu,\nu}_{i+\frac{1}{2},j,k} \Big),
\label{eq:2D:Gauss1}
\end{align}
where $\hat\omega_1 = \hat \omega_{\tt L}$ has been used. Similarly, we have
\begin{equation}\label{eq:2D:Gauss2}
\begin{aligned}
\overline{\bf W}_{ijk}^n  & =  \sum \limits_{\delta = 2}^{ {\tt L}-1}   \hat \omega_\delta  \omega_\mu \omega_\nu     {\bf W}_{ijk}^n\big( {\tt x}_i^{(\mu)},\hat {\tt y}_j^{(\delta)}, {\tt z}_k^{(\nu)}\big)
\\
& \quad
+ \hat \omega_1   \omega_\mu \omega_\nu  \Big(   {\bf W}_{i,j-\frac12,k}^{\mu,+,\nu} + {\bf W}_{i,j+\frac12,k}^{\mu,-,\nu}  \Big),
\end{aligned}
\end{equation}
and
\begin{equation}\label{eq:2D:Gauss3}
\begin{aligned}
\overline{\bf W}_{ijk}^n  &=  \sum \limits_{\delta = 2}^{ {\tt L}-1}  \hat \omega_\delta  \omega_\mu \omega_\nu
 {\bf W}_{ijk}^n\big({\tt x}_i^{(\mu)},{\tt y}_j^{(\nu)},\hat {\tt z}_k^{(\delta)}\big)
 \\
 & \quad
+ \hat \omega_1   \omega_\mu \omega_\nu  \Big(   {\bf W}_{i,j,k-\frac12}^{\mu,\nu,+} + {\bf W}_{i,j,k+\frac12}^{\mu,\nu,-}  \Big).
\end{aligned}
\end{equation}
Taking a weighted average of Eqs. \eqref{eq:2D:Gauss1}--\eqref{eq:2D:Gauss3} gives
\begin{align*}
\begin{split}
& \overline{\bf W}_{ijk}^n    =  \frac{  1 }{  \Delta_\ell^{-1} a_\star^{(\ell)} }
 \Big( \Delta_1^{-1}  a_\star^{(1)} \times \mbox{Eq. \eqref{eq:2D:Gauss1}}
 \\
 & \qquad
 + \Delta_2^{-1}  a_\star^{(2)} \times \mbox{Eq. \eqref{eq:2D:Gauss2}} + \Delta_3^{-1}  a_\star^{(3)} \times \mbox{Eq. \eqref{eq:2D:Gauss3}}    \Big)
\end{split}
\\
\begin{split}
& = (1-2 \hat \omega_1) \Pi_\star
+
 \hat \omega_1  \frac{ \omega_\mu \omega_\nu  }{  \Delta_\ell^{-1} a_\star^{(\ell)} }
  \Big(   \Delta_1^{-1}   a_\star^{(1)}   {\bf W}_{i-\frac12,j,k}^{+,\mu,\nu}
\\
&
~~  +   \Delta_1^{-1}  a_\star^{(1)} {\bf W}_{i+\frac12,j,k}^{-,\mu,\nu}
+ \Delta_2^{-1}  a_\star^{(2)} \big(  {\bf W}_{i,j-\frac12,k}^{\mu,+,\nu} +  {\bf W}_{i,j+\frac12,k}^{\mu,-,\nu}   \big)
\\
&
~~ + \Delta_3^{-1}  a_\star^{(3)} \big(  {\bf W}_{i,j,k-\frac12}^{\mu,\nu,+} +
{\bf W}_{i,j,k+\frac12}^{\mu,\nu,-} \big) \Big) ,
\end{split}
\end{align*}
with $\Pi_\star $ defined by the convex combination
\begin{align*}
\Pi_\star  & =
\frac{1}{1-2 \hat \omega_1}
 \sum \limits_{\delta = 2}^{{\tt L}-1} \bigg\{ \hat \omega_\delta  \times
  \frac{   \omega_\mu \omega_\nu}{  \Delta_\ell^{-1}  a_\star^{(\ell)} }
 \\
 & \quad
 \times  \Big( \Delta_1^{-1}  a_\star^{(1)} {\bf W}_{ijk}^n
  \big(\hat {\tt x}_i^{(\delta)},{\tt y}_j^{(\mu)},{\tt z}_k^{(\nu)}\big)
  \\
& \quad ~~
+ \Delta_2^{-1}  a_\star^{(2)}
     {\bf W}_{ijk}^n\big({\tt x}_i^{(\mu)},\hat {\tt y}_j^{(\delta)},{\tt z}_k^{(\nu)}\big)
     \\
     & \quad ~~
+ \Delta_3^{-1}   a_\star^{(3)} {\bf W}_{ijk}^n\big( {\tt x}_i^{(\mu)}, {\tt y}_j^{(\nu)},\hat {\tt z}_k^{(\delta)}\big) \Big) \bigg\},
\end{align*}
which belongs to ${\mathcal G}_*$ by the hypothesis and the convexity of ${\mathcal G}_*$.
Furthermore, $\Xi_H$ can be reformulated as
\begin{align} \label{eq:XiF}
\Xi_H
=  (1-2 \hat \omega_1) \Pi_\star + 2\hat \omega_1 \widehat{ \Pi}_\star ,
\end{align}
where
\begin{equation*} %\label{eq:hatPistar}
\widehat{ \Pi}_\star =  \frac{\omega_\mu \omega_\nu  \big(
\Delta_1^{-1} a_\star^{(1)}  \Pi^{\mu \nu}_1
+ \Delta_2^{-1} a_\star^{(2)}  \Pi^{\mu \nu}_2 + \Delta_3^{-1} a_\star^{(3)}  \Pi^{\mu \nu}_3
\big)}{  \Delta_\ell^{-1} a_\star^{(\ell)}   }
,
\end{equation*}
with
\begin{equation*}
{\small
\begin{aligned}
 &\Pi^{\mu \nu}_1= \frac{1}{2}\Big(  {\bf W}_{i-\frac12,j,k}^{+,\mu,\nu} + {\bf W}_{i+\frac12,j,k}^{-,\mu,\nu} \Big)  + \frac{\varpi} { 2a^{(1)}_\star}
\\
&
~\times  \Big(
  {\widehat {\bf H}}^1 \big( {\bf W}^{-,\mu,\nu}_{i-\frac{1}{2},j,k}, {\bf W}^{+,\mu,\nu}_{i-\frac{1}{2},j,k} \big)-
   {\widehat {\bf H} }^1 \big( {\bf W}^{-,\mu,\nu}_{i+\frac{1}{2},j,k}, {\bf W}^{+,\mu,\nu}_{i+\frac{1}{2},j,k} \big)
  \Big),
\\
& \Pi^{\mu \nu}_2=   \frac{1}{2}\Big( {\bf W}_{i,j-\frac12,k}^{\mu,+,\nu} + {\bf W}_{i,j+\frac12,k}^{\mu,-,\nu} \Big) + \frac{\varpi} { 2a^{(2)}_\star}
\\
&
~\times \Big(
 {\widehat {\bf H}}^2 \big( {\bf W}^{\mu,-,\nu}_{i,j-\frac{1}{2},k}, {\bf W}^{\mu,+,\nu}_{i,j-\frac{1}{2},k} \big)-
  {\widehat {\bf H}}^2 \big( {\bf W}^{\mu,-,\nu}_{i,j+\frac{1}{2},k}, {\bf W}^{\mu,+,\nu}_{i,j+\frac{1}{2},k} \big)
  \Big),\\
& \Pi^{\mu \nu}_3 =   \frac{1}{2}\Big( {\bf W}_{i,j,k-\frac12}^{\mu,\nu,+} + {\bf W}_{i,j,k+\frac12}^{\mu,\nu,-}  \Big) + \frac{\varpi} { 2a^{(3)}_\star}
\\
&
~\times \Big(
   {\widehat {\bf H}}^3 \big( {\bf W}^{\mu,\nu,-}_{i,j,k-\frac12}, {\bf W}^{\mu,\nu,+}_{i,j,k-\frac12} \big)
  - {\widehat {\bf H}}^3 \big( {\bf W}^{\mu,\nu,-}_{i,j,k+\frac12}, {\bf W}^{\mu,\nu,+}_{i,j,k+\frac12} \big)
  \Big),
\end{aligned}}
\end{equation*}
and $\varpi = \Delta t_n \big(
 \Delta_\ell^{-1} a^{(\ell)}_\ell + \hat \omega_1\lambda_{\rm S} \big)/ \hat \omega_1 \in(0,1) $ under the condition \eqref{eq:FVDGcfl}.
Note that $\Pi^{\mu \nu}_1$ can be rewritten as
\begin{align}\label{eq:app:proof1}
\Pi^{\mu \nu}_1 & = \left(1-   \frac{\varpi}{2} \right) \Pi^{\mu \nu}_{1,1} +  \frac{\varpi}{2} \Pi^{\mu \nu}_{1,2},
\end{align}
where
\begin{align*}
\begin{split}
 \Pi^{\mu \nu}_{1,1} &= \frac{1}{2} \Big( {\bf W}_{i-\frac12,j,k}^{+,\mu,\nu} + \varpi_\star^{-1} {\bf H}^1
 \big( {\bf W}_{i-\frac12,j,k}^{+,\mu,\nu} \big) \Big)
\\
&\quad + \frac{1}{2} \Big( {\bf W}_{i+\frac12,j,k}^{-,\mu,\nu} -  \varpi_\star^{-1} {\bf H}^1
 \big( {\bf W}_{i+\frac12,j,k}^{-,\mu,\nu} \big) \Big),
\end{split}
 \\
 \begin{split}
  \Pi^{\mu \nu}_{1,2} &= \frac{1}{2} \Big( {\bf W}_{i-\frac12,j,k}^{-,\mu,\nu} + \big( a_\star^{(1)} \big)^{-1}  {\bf H}^1
 \big( {\bf W}_{i-\frac12,j,k}^{-,\mu,\nu} \big) \Big)
 \\
 &\quad+ \frac{1}{2} \Big( {\bf W}_{i+\frac12,j,k}^{+,\mu,\nu} - \big( a_\star^{(1)} \big)^{-1}    {\bf H}^1
 \big( {\bf W}_{i+\frac12,j,k}^{+,\mu,\nu} \big) \Big) ,
\end{split}
\end{align*}
with $\varpi_\star= \frac{2-\varpi}{\varpi} a_\star^{(1)} >  a_\star^{(1)}$. With the help of Lemma \ref{lam:LxF} and the convexity of ${\mathcal G}_*$ and
$\overline{\mathcal G}_*$, we have  $\Pi^{\mu \nu}_{1,1} \in {\mathcal G}_* $ and $\Pi^{\mu \nu}_{1,2} \in \overline{\mathcal G}_*$ from \eqref{eq:astarFVDG}.
These further imply
$\Pi^{\mu \nu}_1 \in {\mathcal G}_*$ by Lemma \ref{lam:convex} and \eqref{eq:app:proof1}. Similar arguments yield $\Pi^{\mu \nu}_2,\Pi^{\mu \nu}_3 \in {\mathcal G}_*$.
Using the convexity of ${\mathcal G}_*$ again, we obtain $\hat \Pi_\star \in {\mathcal G}_*$.
From \eqref{eq:XiF} and $\Pi_\star \in {\mathcal G}_*$, we draw the conclusion $\Xi_H \in {\mathcal G}_*$ based on the convexity of ${\mathcal G}_*$.

(2). Then prove $\Xi_S \in \overline{\mathcal G}_*$ by separately considering two cases.
\begin{itemize}[\hspace{0em}$\bullet$]
  \item If $q\big( \overline{ {\bf S}}_{ijk}^n \big) \ge 0$, then $\lambda_{\rm S}=0$ and $\vartheta=1$,
which yield $\Xi_S = \Delta t_n \overline{ {\bf S}}_{ijk}^n$. Because the first component of $\Xi_S$ is zero and $q(\Xi_S)=\Delta t_n q\big( \overline{ {\bf S}}_{ijk}^n \big) \ge 0$, we thus have $\Xi_S \in \overline{\mathcal G}_*$.
  \item If $q\big( \overline{ {\bf S}}_{ijk}^n \big) < 0$, then $\overline{ {\bf S}}_{ijk}^n  \notin {\mathcal G}_*$. Thanks to the convexity of ${\mathcal G}_*$, Eq. \eqref{eq:FVDGs} has and only has one positive solution, which is $\lambda_{\rm S}>0$.
      This further implies
      $ \overline{\bf W}_{ijk}^n + \lambda \overline{ {\bf S}}_{ijk}^n \in {\mathcal G}_* $ for any $\lambda \in [0,\lambda_{\rm S}^{-1})$.
      Specially,  $ \overline{\bf W}_{ijk}^n + \varpi \lambda_{\rm S}^{-1} \overline{ {\bf S}}_{ijk}^n \in {\mathcal G}_* $. It follows from the scaling invariance of ${\mathcal G}_* $ that
      $$ \Xi_S =\varpi^{-1} \Delta t_n \lambda_{\rm S} \big( \overline{\bf W}_{ijk}^n + \varpi \lambda_{\rm S}^{-1} \overline{ {\bf S}}_{ijk}^n \big) \in {\mathcal G}_* \subset \overline{\mathcal G}_*.$$
\end{itemize}

(3). Using the deductions proved in parts (1) and (2),
we respectively obtain $2 \vartheta \Xi_H \in {\mathcal G}_*$ and $2\Xi_S\in \overline {\mathcal G}_*$, based on
the scaling invariance of ${\mathcal G}_* $.
Thanks to Lemma \ref{lam:convex}, it holds that
$$\overline {\bf W}_{ijk}^{n+1} = \frac{1}{2} \cdot 2 \vartheta \Xi_H  + \frac{1}{2} \cdot 2\Xi_S \in {\mathcal G}_*.$$

The proof is completed.
\end{proof}

\subsection{\label{app:6}Proof of Lemma \ref{lem:twosets}}

\begin{proof}
Eq. \eqref{eq:defWtheta} implies that ${\bf W}_{i,j,k}(\bm \theta)$ is linear with respect to $\bm \theta$. Hence $ \Theta_0 $ is a convex set.
Assume that $\bm \theta_0,\bm \theta_1 \in \Theta$. Then, for any $\lambda \in [0,1]$, we have $\bm \theta_\lambda  = (1-\lambda) \bm \theta_0 + \lambda  \bm \theta_1 \in \Theta_0 $. The concavity of the function $q({\bf W})$ yields
\begin{align*}
q\big( {\bf W}_{i,j,k} (\bm \theta_\lambda ) \big)  & =
q \big( (1-\lambda)  {\bf W}_{i,j,k} ( \bm \theta_0 ) + \lambda  {\bf W}_{i,j,k} ( \bm \theta_1 ) \big) \\
& \ge (1-\lambda) q\big( {\bf W}_{i,j,k} (\bm \theta_0 ) \big) + \lambda q\big( {\bf W}_{i,j,k} (\bm \theta_1 ) \big)
\\
&\ge (1-\lambda) \epsilon + \lambda \epsilon = \epsilon.
\end{align*}
%If follows that $\bm \theta_\lambda \in
%\Theta$ for any $\lambda \in [0,1]$. Therefore $\Theta$ is also convex.
If follows that $\bm \theta_\lambda \in
\Theta,~\forall \lambda \in [0,1]$. Hence $\Theta$ is convex.
\end{proof}

\end{document}